\documentclass{article}
\usepackage[utf8]{inputenc}
\usepackage{geometry}
\newgeometry{hmargin={2.6cm,2.6cm},vmargin={2.5cm,2.5cm}}   
\usepackage{color,amssymb,amsmath,epsfig,ifthen,graphicx,tabularx,xargs}
\usepackage{amsfonts}
\usepackage{amsthm}
\usepackage{algorithm}
\usepackage{algpseudocode}
\usepackage{enumitem,verbatim}
\usepackage{hyperref}
\hypersetup{colorlinks=true,citecolor=blue, urlcolor=blue, linkcolor=black, filecolor=black}
\usepackage{cleveref}
\newtheorem{theorem}{Theorem}[section]
\crefname{theorem}{theorem}{theorems}
\Crefname{Theorem}{Theorem}{Definitions}

\newtheorem{definition}{Definition}[section]
\crefname{definition}{definition}{definitions}
\Crefname{Definition}{Definition}{Definitions}

\newtheorem{corollary}[theorem]{Corollary}
\crefname{corollary}{corollary}{corollaries}
\Crefname{Corollary}{Corollary}{Corollaries}

\newtheorem{proposition}[theorem]{Proposition}
\crefname{proposition}{proposition}{propositions}
\Crefname{Proposition}{Proposition}{Propositions}

\newtheorem{lemma}[theorem]{Lemma}
\crefname{lemma}{lemma}{lemmas}
\Crefname{Lemma}{Lemma}{Lemmas}

\crefname{example}{example}{examples}
\Crefname{example}{Example}{Examples}

\crefname{remark}{remark}{remarks}
\Crefname{remark}{Remark}{Remarks}

\newtheorem{assumption}{Assumption}[section]
\crefname{assumption}{assumption}{assumptions}
\Crefname{assumption}{Assumption}{Assumptions}

\crefname{enumi}{point}{point}
\Crefname{enumi}{Point}{Point}

\theoremstyle{remark}
\newcommand{\ceil}[1]{{\left\lceil #1 \right\rceil}}

\newcommand{\norm}[1]{{
    \left\| #1 \right\|}} 

\newcommandx\cspanarg[2][1=]{\ensuremath{\overline{\mathrm{Span}}^{#1}\left(#2\right)}}

\newcommand{\spec}{{\rm \spec}}

\newcommand{\tnorm}[1]{{\left\vert\kern-0.25ex\left\vert\kern-0.25ex\left\vert #1 
    \right\vert\kern-0.25ex\right\vert\kern-0.25ex\right\vert}} 

\newcommandx{\fiarmapred}[1][1={\mapol,\arpol,D}]{\mathrm{\Phi}^{\dagger}_{#1}}
\newcommandx{\fiarmapredcoef}[2][1={\mapol,\arpol,D}]{\aop^{\dagger}_{#2}\lr{#1}}

\newcommand{\smallo}[1]{o \left( #1 \right)}
\newcommand{\bigo}[1]{\mathcal{O} \left( #1 \right)}

\def\rset{\mathbb{R}}

\def\zset{\mathbb{Z}}
\def\nset{\mathbb{N}}

\usepackage{amsfonts}
\usepackage[bbgreekl]{mathbbol}
\DeclareSymbolFontAlphabet{\mathbbm}{bbold}
\DeclareSymbolFontAlphabet{\mathbb}{AMSb}%
\def\mapol{\bbtheta}
\def\arpol{\bbphi}

  {\begin{enumerate}}%
  {\end{enumerate}}

\def\rme{\mathrm{e}}

\newcommandx{\aslim}[1]{\ensuremath{\stackrel{#1\text{a.s.}}{\longrightarrow}}}  

\newcommand{\1}{\mathbbm{1}}

\def\bX{\mathbf{X}}

\def\eqsp{\;}

\newcommand{\aop}{\mathrm{P}}

\def\cF{\mathcal{F}}
\def\cE{\mathcal{E}}

\def\cB{\mathcal{B}}

\def\cH{\mathcal{H}}
\def\cG{\mathcal{G}}

\def\Xset{\mathsf{X}}
\def\Yset{\mathsf{Y}}

\newcommand{\pscal}[2]{\left\langle #1, #2 \right\rangle}

\def\b0{{\bf 0}}

\def\bX{{\bf X}}

\def\cA{\mathcal{A}}
\def\cB{\mathcal{B}}
\def\cC{\mathcal{C}}

\def\cE{\mathcal{E}}
\def\cF{\mathcal{F}}
\def\cG{\mathcal{G}}
\def\cH{\mathcal{H}}

\def\cL{\mathcal{L}}

\def\cR{\mathcal{R}}
\def\cS{\mathcal{S}}

\newcommand{\argmin}{\mathop{\mathrm{argmin}}}

\newcommand\algo[1]%
{
    \begin{center} %
    \begin{tabular} {||p{10 cm}l ||}%
               #1 &  \\
    \hline
    \end{tabular}
    \end{center}
}

\newcommand{\Ysigma}{\ensuremath{\mathcal{Y}}}

\newcommand{\chunk}[4][]%
{\ifthenelse{\equal{#1}{}}{\ensuremath{{#2}_{#3:#4}}}{\ensuremath{#2^#1}_{#3:#4}}
}

\def\esp{\mathbb{E}}
\newcommandx\prob[2][1=,2=]{\ensuremath{{\mathbb P}_{#1}^{#2}}}
\newcommand{\PP}[1][]{\ifthenelse{\equal{#1}{}}{\ensuremath{\mathbb{P}}}{\ensuremath{\mathbb{P}\left( #1 \right)}}}
\newcommandx{\PParg}[2][1=]{\PP_{#1}\left(#2\right)}
\newcommand{\PE}[1][]{\ifthenelse{\equal{#1}{}}{\esp}{\ensuremath{{\mathbb E}\left[ #1 \right]}}}
\newcommandx{\PEarg}[2][1=]{\PE_{#1}\left[#2\right]}
\newcommand{\PVar}{\ensuremath{\operatorname{Var}}}
\newcommandx\var[2][1=]{\ensuremath{\PVar_{#1}\left( #2\right)}}
\newcommandx\cvar[3][1=]{\ensuremath{\PVar_{#1}\left( \left. #2 \right| #3 \right)}}

\newcommandx\cov[3][1=]{\ensuremath{\mathrm{Cov}_{#1}\left( #2,#3 \right)}}
\newcommandx\ccov[3][1=]{\ensuremath{\mathrm{Cov}_{#1}\left( \left. #2 \right| #3 \right)}}

\newcommand{\CPP}[3][]
{\ifthenelse{\equal{#1}{}}{\PP\left(\left. #2 \, \right| #3 \right)}{\mathbb{P}_{#1}\left(\left. #2 \, \right | #3 \right)}}
\newcommand{\CPE}[3][]
{\ifthenelse{\equal{#1}{}}{\PE\left[ \left. #2 \right| #3
    \right]}{\mathbb{E}_{#1} \left[ \left. #2 \right| #3 \right]}}
\newcommandx\cprob[4][1=,2=]{\ensuremath{\PP_{#1}^{#2}\left( \left. #3 \right|
      #4 \right)}}
\newcommandx\HMCP[2][1=]{
\ifthenelse{\equal{#1}{}}{\PP_{#2}}{\PP_{#1,#2}}
}
\newcommandx\HMCE[2][1=]{
\ifthenelse{\equal{#1}{}}{\PE_{#2}}{\PE_{#1,#2}}
}
\newcommandx\HMCEarg[3][1=]{
\ifthenelse{\equal{#1}{}}{\PE_{#2}\left[#3\right]}{\PE_{#1,#2}\left[#3\right]}
}

\def\loikhi2{\mathbf{\chi^2}}

\newcommandx{\proj}[2]{\ensuremath{\operatorname{proj}\left( \left. #1\right|#2\right)}}

\newcommand{\URoot}{\ensuremath{R}}
\newcommand{\UCov}[1][]%
{%
\ifthenelse{\equal{#1}{}}{\URoot \URoot^t}{\URoot_{#1} \URoot^t_{#1}}%
}

\newcommand{\VRoot}{\ensuremath{S}}
\newcommand{\VCov}[1][]%
{%
\ifthenelse{\equal{#1}{}}{\VRoot \VRoot^t}{\VRoot_{#1} \VRoot^t_{#1}}%
}

\newcommand{\LDX}[2]{\ensuremath{L}}

\newcommand{\postdx}[3][]%
{%
\ifthenelse{\equal{#1}{}}{\ensuremath{\psi_{#2|#3}}}{\ensuremath{\psi_{#1,#2|#3}}}%
}

\newcommand{\epostdx}[3][]%
{%
\ifthenelse{\equal{#1}{}}{\ensuremath{\hat{\psi}_{#2|#3}}}{\ensuremath{\hat{\psi}_{#1,#2|#3}}}%
}

\newcommandx{\predx}[3][1=\bX]{#1_{#2|#3}}   
\newcommand{\predpx}[3][]%
{%
\ifthenelse{\equal{#1}{}}{\ensuremath{\varphi_{#2|#3}}}{\ensuremath{\varphi_{#1,#2|#3}}}%
}
\newcommandx\cesp[4][1=,2=]{\ensuremath{{\mathbb E}_{#1}^{#2}\left[ \left. #3 \right| #4 \right]}}

\newcommand{\filt}[2][]%
{%
\ifthenelse{\equal{#1}{}}{\ensuremath{\phi_{#2}}}{\ensuremath{\phi_{#1,#2}}}%
}
\newcommand{\pred}[3][]%
{%
\ifthenelse{\equal{#1}{}}{\ensuremath{\phi_{#2|#3}}}{\ensuremath{\phi_{#1,#2|#3}}}%
}
\newcommand{\post}[3][]%
{%
\ifthenelse{\equal{#1}{}}{\ensuremath{\phi_{#2|#3}}}{\ensuremath{\phi_{#1,#2|#3}}}%
}
\newcommand{\logl}[2][]%
{%
\ifthenelse{\equal{#1}{}}{\ensuremath{\ell_{#2}}}{\ensuremath{\ell_{#1,#2}}}%
}
\newcommand{\lhood}[2][]%
{%
\ifthenelse{\equal{#1}{}}{\ensuremath{\mathrm{L}_{#2}}}{\ensuremath{\mathrm{L}_{#1,#2}}}%
}
\newcommand{\cc}[2][]%
{%
\ifthenelse{\equal{#1}{}}{\ensuremath{c_{#2}}}{\ensuremath{c_{#1,#2}}}%
}
\newcommand{\forvar}[2][]%
{%
\ifthenelse{\equal{#1}{}}{\ensuremath{\alpha_{#2}}}{\ensuremath{\alpha_{#1,#2}}}%
}
\newcommand{\nforvar}[2][]%
{%
\ifthenelse{\equal{#1}{}}{\ensuremath{\bar{\alpha}_{#2}}}{\ensuremath{\bar{\alpha}_{#1,#2}}}%
}

\newcommand{\BK}[2][]%
{%
\ifthenelse{\equal{#1}{}}{\ensuremath{\mathrm{\mathrm{B}}_{#2}}}{\ensuremath{\mathrm{B}_{#1,#2}}}%
}
\newcommand{\filtfunc}[2][]%
{%
\ifthenelse{\equal{#1}{}}{\ensuremath{\tau_{#2}}}{\ensuremath{\tau_{#1,#2}}}%
}

\newcommand{\filtmean}[2][]
{\ifthenelse{\equal{#1}{}}{{\ensuremath{\hat{X}_{#2|#2}}}}{\ensuremath{\hat{X}_{#1,#2|#2}}}
}
\newcommand{\filtcov}[2][]
{\ifthenelse{\equal{#1}{}}{\ensuremath{\Sigma_{#2|#2}}}{\ensuremath{\Sigma_{#1,#2|#2}}}}
\newcommand{\postmean}[3][]
{\ifthenelse{\equal{#1}{}}{\ensuremath{\hat{X}_{#2|#3}}}{\ensuremath{\hat{X}_{#1,#2|#3}}}
}
\newcommand{\postcov}[3][]
{\ifthenelse{\equal{#1}{}}{\ensuremath{\Sigma_{#2|#3}}}{\ensuremath{\Sigma_{#1,#2|#3}}}}

\newcommandx{\QEM}[4][1=,4=]{\ensuremath{\mathcal{Q}_{#1}(#4;#2 \, ; #3)}}

\newcommand{\argmax}{\ensuremath{\mathop{\mathrm{argmax}}}}

\newcommandx\sequence[3][2=t,3=\zset]{\ensuremath{\left(#1_{#2}\right)_{#2 \in #3 }}}
\newcommandx\dsequence[4][3=t,4=\zset]{\ensuremath{\left( (#1_{#3}, #2_{#3})\right)_{#3 \in #4}}}
\newcommandx{\sequencen}[2][2=n\in\nset]{\ensuremath{\left(#1\right)_{#2}}}

\def\bpm{\left[\begin{matrix}}
\def\epm{\end{matrix}\right]}
\def\bma{\begin{matrix}}
\def\ema{\end{matrix}}

\newcommand{\be}{\begin{equation}}     
\newcommand{\ee}{\end{equation}}        

\newcommandx\lnorm[3][1=]{\left\lVert #2 \right\rVert^{#1}_{#3}}

\newcommandx\supnorm[2][1=]{| #2 |^{#1}_\infty}

\newcommandx\ball[3][1=]{\mathrm{B}_{#1} (#2,#3)}

\newcommandx{\prohosym}[1][1=]{{\boldsymbol\rho}_{#1}}
\newcommandx{\proho}[3][1=]{\prohosym{#1}\left(#2,#3\right)}

\newcommandx{\pp}[1][1=\mu]{\ensuremath{#1\-\mathrm{a.e.}}}

\renewcommand{\-}{\mbox{-}}

\newcommandx{\as}[1][1=\PP]{\ensuremath{#1\-\mathrm{a.s.}}}

\newcommandx{\oscnorm}[3][1=,3=]{\operatorname{osc}^{#1}_{#3}\left(#2\right)}

\newcommandx{\tvdist}[3][1=]{\ensuremath{d^{#1}_{\mathrm{TV}}}(#2,#3)}

\newcommandx{\VnormFunc}[3][1=]{\ensuremath{\left|#2\right|_{\mathrm{#3}}^{#1}}}

\newcommand{\continuousfunctionset}[1]{\mathrm{C}_b(#1)}

\newcommand{\lipschitzfunctionset}[1]{\mathrm{Lip}(#1)}
\newcommand{\boundedlipschitzfunctionset}[1]{\mathrm{Lip}_b(#1)}

\newcommandx\functionsetarg[2][1=]{
\ifthenelse{\equal{#1}{c}}{\continuousfunctionset{\mathsf{#2}}}
{\ifthenelse{\equal{#1}{bc}}{\mathrm{C}_b(#2)}
{\ifthenelse{\equal{#1}{u}}{\mathrm{U}(#2)}
{\ifthenelse{\equal{#1}{bu}}{\mathrm{U}_b(#2)}
{\ifthenelse{\equal{#1}{l}}{\lipschitzfunctionset{#2}}
{\ifthenelse{\equal{#1}{bl}}{\boundedlipschitzfunctionset{#2}}
{\mathbb{F}_{#1}(#2)}
}}}}}}

\newcommandx\functionsetspec[1][1=]{
\ifthenelse{\equal{#1}{c}}{\mathrm{C}}
{\ifthenelse{\equal{#1}{bc}}{\mathrm{C}_b}
{\ifthenelse{\equal{#1}{u}}{\mathrm{U}}
{\ifthenelse{\equal{#1}{bu}}{\mathrm{U}_b}
{\ifthenelse{\equal{#1}{l}}{\mathrm{Lip}}
{\ifthenelse{\equal{#1}{bl}}{\mathrm{Lip}_b}
{\mathbb{F}_{#1}}
}}}}}}

\newcommandx{\taboo}[3][1=,3=]{\left(\leftidx{_#1}{#2}{}\right){^{#3}}}

\newcommandx\vectornorm[2][1=]{\left| #2 \right|^{#1}}  

\newcommand{\ensemble}[2]{\left\{#1\,:\eqsp #2\right\}}
\newcommand{\set}[2]{\ensemble{#1}{#2}}

\newcommandx{\plim}[1]{\ensuremath{\stackrel{#1\-\text{prob}}{\longrightarrow}}}
\newcommandx{\dlim}[1]{\ensuremath{\stackrel{#1}{\Longrightarrow}}}

\newcommandx\measureset[3][1=\mathrm{s},3=]{\mathbb{M}^{#3}_{#1}(#2)}
\newcommandx\measuresetmetric[2][1=1]{\mathbb{M}_{#1}(\mathcal{B}(\mathsf{#2}))}  
\newcommandx\measuresetspec[1][1=\mathrm{s}]{\mathbb{M}_{#1}}

\newcommand{\abs}[1]{\left\vert #1 \right\vert}

\newcommandx\canonicalkernel[1][1=P]{\mathbb{K}_{#1}}

\newcommand{\lr}[1]{\left(#1 \right)}

\newcommand{\lrc}[1]{\left\{#1 \right\}}

\usepackage{mathrsfs}
\usepackage{stmaryrd}
\usepackage{caption}
\usepackage{subcaption}
\usepackage{pdflscape}
\usepackage{authblk}

\usepackage[textwidth=2cm,textsize=footnotesize]{todonotes}    

\newcommand{\linlog}{\,\mathrm{linlog}}

\title{Power comparison of sequential testing by betting procedures.}
\begin{document}
\author[1,2]{Amaury Durand}
\author[2]{Olivier Wintenberger}
\affil[1]{\'Electricit\'e de France R\&D, Bd Gaspard Monge, 91120 Palaiseau, France. }
\affil[2]{Sorbonne Universit\'e, 4 Place Jussieu, 75005 Paris, France. }

\maketitle
\begin{abstract}
In this paper, we derive power guarantees of some sequential tests for bounded mean under general alternatives. We focus on testing procedures using nonnegative supermartingales which are anytime valid and consider alternatives which coincide asymptotically with the null (e.g. vanishing mean) while still allowing to reject in finite time. Introducing variance constraints, we show that the alternative can be broaden while keeping power guarantees for certain second-order testing procedures. We also compare different test procedures in multidimensional setting using characteristics of the rejection times. Finally, we extend our analysis to other functionals as well as testing and comparing forecasters.   
Our results are illustrated with numerical simulations including bounded mean testing and comparison of forecasters. 
\end{abstract}

\section{Introduction}
Safe anytime valid testing provides tests that remain valid at all stopping times thus allowing for optional stopping or continuation. This property guarantees that we can collect data sequentially and decide to reject or not the null $\cH_0$ at each time step without compromising the level of the test. More precisely,  given a level $\alpha\in (0,1)$, a safe anytime valid test provides a decision in the form of a  \emph{rejection time} $\tau_\alpha \in \nset\cup\{+\infty\}$ satisfying
\begin{equation}\label{eq:reject-level}
\PP[\tau_\alpha < +\infty] \leq \alpha ,\; \text{under } \cH_0\,,
\end{equation}
hence controlling the level of the test. Another desirable property is that the test is of \emph{power one} under an appropriate alternative $\cH_1$, namely 
\begin{equation}\label{eq:power-one}
\PP[\tau_\alpha < +\infty] = 1 ,\; \text{under } \cH_1\,. 
\end{equation}
In a parametric context, this type of test has been constructed using likelihood ratio sequences (see e.g. \cite{Wald45,Wald_optimum}).
Generalizations to non-parametric cases were considered in \cite{Darling_powerone} and more recently using the notions of test supermartingales \cite{Shafer11_test_martingales} or e-processes \cite{Grunwald19-safe-testing}. Given a probability space $(\Omega, \cF,\PP)$ endowed with a filtration $(\cF_t)_{t\in\nset}$, we recall the definition of a test supermartingale. 
\begin{definition}\label{def:test-supermartingale}
A test supermartingale for a null hypothesis $\cH_0$ is a process $(W_t)_{t\in\nset}$ with $W_0 = 1$ and such that, for all $t\geq 1$, $W_t \geq 0$ and $\CPE{W_t}{\cF_{t-1}} \leq W_{t-1}$ under $\cH_0$. If the last inequality is an equality we say that the process is a test martingale. 
\end{definition}

\cite{Shafer21-test-betting} developed a nice betting interpretation for test supermartingales. Namely, starting with a capital (or wealth) of $W_0 = 1$, we bet against the null hypothesis and observe how the wealth $(W_t)_{t\in\nset}$ evolves over time. A test supermartingale indicates that we expect to loose under the null and a test martingale indicates that we are in a fair game. 
 Given an appropriate betting strategy, our capital should grow if we accumulate enough evidence against the null and decrease otherwise. A consequence of Ville's theorem \cite{ville1939etude} is that the rejecting time 
\begin{equation}\label{eq:def-tau}
\tau_\alpha := \inf\set{t\in\nset}{W_t \geq 1/\alpha}\,,
\end{equation}
satisfies \eqref{eq:reject-level} if $(W_t)_{t\in\nset}$ is a nonnegative supermartingale. In the following, we will focus on tests of the form \eqref{eq:def-tau} and assert power guarantees thanks to stochastic properties (finiteness and first-moment's  bound) of $\tau_\alpha$. 

\subsection{Related works}
The literature on safe anytime valid inference (SAVI), which includes tests and confidence sequences, has been rapidly growing in recent years and we refer the reader to \cite{Ramdas22-survey} and \cite{ramdas2024hypothesistestingevalues} for recent surveys. One of the key tools to derive safe anytime valid confidence sequences is time-uniform concentration bounds which are thoroughly studied in \cite{Howard20-chernoff}. In \cite{Howard21-CS} the authors provide a general framework to construct safe anytime confidence sequences with vanishing width using stitching methods or nonnegative martingale mixtures. In \cite{WaudbySmith20_meanbetting}, the authors construct such confidence sequences for the mean of bounded variables using betting strategies.  The case of unbounded means with bounded variances is studied in \cite{Wang23-CS-unbounded}. These ideas have been extended to the estimation of other quantities than the mean. See, for example, \cite{Howard22-quantileCS} for quantiles, \cite{Manole_2023} for the estimation of convex divergences between two distributions and \cite{Choe21-comparing-forecasters} for the average score difference between two forecasters. 
There is a strong link between safe anytime valid confidence sequences and tests since, to test a null stating that the quantity of interest is equal to $\mu$, one can reject the null as soon as $\mu$ is not in the confidence sequence. Note that this test is however not of the form \eqref{eq:def-tau}. Other contributions propose tests of the form \eqref{eq:def-tau} based on test supermartingale or e-processes. For example, the confidence sequences derived in \cite{WaudbySmith20_meanbetting,Wang23-CS-unbounded} rely on test supermartingales and \cite{Choe21-comparing-forecasters} also propose an e-process to test whether a forecaster outperforms another one on average thus weakening the null hypothesis of \cite{Henzi21-comparing} which tests if one forecaster always outperforms another one. Other works provide tests for a large set of tasks including elicitable and identifiable functionals \cite{cassgrain22-anytime}, forecast calibration \cite{Arnold21-calibration}, Value-at-Risk and Expected Shortfall backtesting \cite{Wang22Ebacktesting}, equality in distribution of two samples \cite{Shekhar21}, testing if the data are drawn i.i.d. from a log-concave distribution \cite{Gangrade23-logconcavity} or exchangeability in the data \cite{RAMDAS22-exchangeability}.

\subsection{Predictable plug-in test supermartingales}\label{sec:plugin}
Given a collection $\set{(L_t(\lambda))_{t\in\nset}}{\lambda\in\Lambda}$ of test supermartingales  for some set $\Lambda\subset\rset^d$, one can show that, 
for any predictable sequence $(\lambda_t)_{t\geq 1}$ valued in $\Lambda$ (referred to as the \emph{betting strategy}), the process defined by $W_0 = 1$ and 
$$
W_t = \prod_{i=1}^{t} \frac{L_i(\lambda_i)}{L_{i-1}(\lambda_i)}\,, \quad t \geq 1 \,,
$$
is also a test supermartingale known as a \emph{predictable plug-in} test supermartingales. This holds also for predictable mixtures, see \cite[Lemma~2.4]{cassgrain22-anytime}. Different strategies to tune the sequence of parameters $(\lambda_t)_{t\geq 1}$ provide different guarantees. For example, in \cite{WaudbySmith20_meanbetting}, the parameters are tuned to control the width of the confidence sequence. Other works use the GRO criterion of \cite{Grunwald19-safe-testing} and select the parameter $\lambda_{t+1}$ which maximizes the growth rate 
\begin{equation}\label{eq:GRO}
\CPE[Q]{\log \frac{L_{t+1}(\lambda)}{L_{t}(\lambda)}}{\cF_{t}}\,,
\end{equation}
for some appropriate distribution $Q$. This growth rate is closely linked to the notion of $e$-power introduced in \cite{vovk_wang_2024}. Intuitively, maximizing the the growth rate should provide optimal power guarantees under the alternative $Q$ and thus it is an appropriate notion of power (see also \cite[Section~2.7]{ramdas2024hypothesistestingevalues} for more formal justification). The problem lies in the choice of $Q$ which should rely on a priori assumptions on the alternative. In our work, we avoid chosing $Q$ by taking the empirical distriution. This is related to the GREE method of \cite{Wang22Ebacktesting} and the GRAPA method of \cite{WaudbySmith20_meanbetting} and is used, for example, in \cite{cassgrain22-anytime}. Hence, we select $\lambda_{t+1}$ by maximizing 
\begin{equation}\label{eq:GREE}
\sum_{i=1}^{t}\log \frac{L_i(\lambda)}{L_{i-1}(\lambda)}  = \log L_{t}(\lambda)
\end{equation}
using an Online Convex Optimization  (OCO) method \cite{hazan2022introduction}, assuming that $\lambda \mapsto \log\frac{L_i(\lambda)}{L_{i-1}(\lambda)}$ is concave, as suggested in \cite{cassgrain22-anytime}.  This means that we do not necessarily take $\lambda_t$ as a maximizer of \eqref{eq:GREE}, which is known as the Follow The Leader (FTL) algorithm, but are interested in using a strategy that provides guarantees on how $W_t$ grows. Typically, this is achieved by controlling the \emph{regret} which is well studied in the OCO literature and writes, in our context, as 
$$
\max_{\lambda\in\Lambda} \log L_t(\lambda) - \log W_t.
$$
The idea behind this strategy is that, if we manage to upper bound the regret of the predictable updates $(\lambda_t)_{t\geq 1}$ of an OCO algorithm, we get a lower bound on $\log W_t$ which can be used to provide guarantees on $\tau_\alpha$ defined by \eqref{eq:def-tau} under an appropriate alternative. 

\subsection{Deriving power guaranteees}
Among the SAVI literature, some works provide theoretical power guarantees which can take three different forms: asymptotic power as in \eqref{eq:power-one}, an asymptotic growth rate for $\log W_n$ or a bound on  $\PE[\tau_\alpha]$ under some alternative. Each of these three power guarantees is more informative than the previous. The most common power guarantee is the asymptotic power, see e.g. \cite{Wang22Ebacktesting}, \cite{Shekhar21}, \cite{cassgrain22-anytime}, \cite{pandeva2023deepanytimevalidhypothesistesting}. However in specific cases, growth rates for $\log W_n$ can be obtained, see e.g. \cite{podkopaev2023sequentialpredictivetwosampleindependence}, \cite{saha24bexchangeability}, \cite{podkopaev23akernelizeiindependence} and even finite bounds for $\PE[\tau_\alpha]$, see e.g. \cite{RobbinsLIL}, \cite{Robbins74expectedsamplesize}, \cite{Shekhar21}, \cite{Chugg23testfairness}. In general, more informative power guarantees are derived under more restrictive alternatives. For example, in \cite{Shekhar21}, the authors provide bounds for $\PE[\tau_\alpha]$ in the i.i.d. case while showing only asymptotic power in a time-varying setting.
In the present work, we show that, for some well constructed test supermartingales such garantees can be obtained under relatively large alternatives. To do so, we rely on a simple (yet often implying tedious calculation) methodology. Namely under a given alternative, we derive a deterministic lower bound for $\log W_n$ and evaluate when this lower bound eventually reaches the desired threshold $\log(1/\alpha)$ thus providing the three aforementioned power guarantees using the following lemma.

\begin{lemma}\label{lem:power_from_bound}
Let $(W_n)_{n\geq 1}$ and $(u_n)_{n\geq 1}$ be respectively be a nonnegative stochastic process and a deterministic sequence satisfying
\begin{equation}\label{eq:boundW}
\varrho := \sum_{n \geq 1} \PP[\log W_n < u_n] < +\infty \;.
\end{equation}
Then
$$
\liminf_{n\to +\infty} \frac{\log W_n}{u_n} \geq 1  \;\;\PP\text{-}a.s.
\quad \text{and} \quad 
\PE[\tau_\alpha] \leq \varrho + \aleph((u_n)_{n\geq 1}, \log(1/\alpha)) \;,
$$
where $\tau_{\alpha}$ is defined in \eqref{eq:def-tau} and we define $\aleph((u_n)_{n\geq 1},x) := \inf \set{n \geq 1}{\inf_{k\geq n} u_k \geq x}$. 
\end{lemma}
\begin{proof}
    The first inequality is a consequence of the Borell-Cantelli theorem and the second comes from  the relation $\PE[\tau_\alpha] = \sum_{n\geq 1} \PP[\tau_\alpha > n] 
\leq \sum_{n\geq 1} \PP[\log W_n <  \log\lr{1/\alpha}]$.
\end{proof}

We focus on showing that \eqref{eq:boundW} holds under some alternatives in a time-varying setting where the distribution of the observation changes over time and converges to the null. For example, in the setting of bounded mean testing, given a real process $(X_t)_{t\in \nset}$, one can be interested to characterize how fast $\PE[X_t]$ can vanish while still allowing our test procedure to reject the null hypothesis stating that $(X_t)_{t\in\nset}$ is centered. This is reminiscent of the notions of \emph{asymptotic power} and \emph{asymptotic relative efficiency} detailed in \cite{noether-pitman} and based on works by Pitman, where the authors are interested in characterizing the asymptotical behavior of the power of a test when the alternative converges to the null as the sample size grows.  The rate of convergence of the alternative to the null can be seen as a \emph{detection boundary} for the test procedure, as discussed in \cite[Remark~11]{Shekhar21}. Comparing the first moments of the rejection times for sequential tests with the same detection boundary is motivated in \cite{lai1978pitman}. Our context is similar but, instead of assuming i.i.d observations and an alternative converging to the null at a rate linked to some stochastic properties of the rejection times, we assume time-varying observations where the marginal distribution of the process converge to a distribution satisfying the null hypothesis. In this context we were only able to obtain upper-bounds on first moments of rejection times. Therefore we run several experiments to discuss the sharpness of our bounds and empirically challenge comparisons based on them. 
Our main results on bounded mean testing are gathered in \Cref{sec:mean-test} and some extentions are discussed in \Cref{sec:extensions}. In \Cref{sec:application} and \Cref{sec:simulation}, we provide some applications and numerical simulations. Proofs are postponed in the supplementary material.

\section{Bounded mean hypothesis testing}\label{sec:mean-test}
In this section, we study hypothesis testing for bounded means and, in particular, two types of sequential testing procedures with non-asymptotic power guarantees.  The first one relies on an exponential test supermartingale based on Hoeffding's lemma 
as proposed in \cite[Section~3.1]{WaudbySmith20_meanbetting}
and the second one corresponds to the \emph{capital process} of \cite{WaudbySmith20_meanbetting} which is also known as the wealth of \emph{coin betting} \cite{OrabonaCoinBetting2016}. 
Let $(\Omega,\cF,\PP)$ be a probability space endowed with a filtration $(\cF_t)_{t\in\nset}$. Given an $(\cF_t)_{t\in\nset}$-adapted process $(X_t)_{t\in\nset}$ valued in a subset $\Xset$ of $\rset^d$ for some $d\geq 1$, we  are interested in testing the null hypothesis. 
\begin{equation}\label{eq:null-X}
\cH_0 \;:\; \PEarg[t-1]{X_t} = 0 \;\; \PP\text{-a.s.}\, \text{ for all } t\in\nset \;, 
\end{equation}
where we use the notation $\PEarg[t-1]{\cdot}=\CPE{\cdot}{\cF_{t-1}}$. Throughout this section, we consider the following assumption.
\begin{assumption}\label{ass:bounded-y} The set $\Xset$ is bounded and we denote $D := \sup_{(x,y)\in\Xset^2} \norm{x-y}_2$ and $B := \sup_{x\in\Xset} \norm{x}_{2}$.
\end{assumption}
We also define, for $n \geq 1$ and $p\in \rset_+ \cup \lrc{\infty}$,
 \begin{equation}\label{eq:def-means}
    \mu_n := \frac{1}{n}\sum_{t=1}^n \PEarg[t-1]{X_t} 
    \quad \text{and} \quad 
    \nu_{n,p} := \frac{1}{n}\sum_{t=1}^n \PEarg[t-1]{\norm{X_t}_p^2}\,.
\end{equation}
We also denote $\cB_r^d := \set{x\in\rset^d}{\norm{x}_2 \leq r}$ for any $d \geq 1$ and $r>0$ and $\linlog(z) := z\log(z)$ for any $z>0$.
Finally, we let  $(m_n)_{n \geq 1}$ and $(v_n)_{n\geq 1}$ be two nonnegative sequences and consider the alternatives hypotheses 
\begin{equation}\label{eq:H1}
\cH_1 : \varrho_1 :=\sum_{n\geq 1} \PP[\norm{\mu_n}_2 < m_n] < +\infty \;,   
\end{equation}
and for $p \in \lrc{1,\cdots,+\infty}$,
\begin{equation}\label{eq:H2}
\cH_{2,p} : \varrho_{2,p} := \sum_{n \geq 1} \PParg{\norm{\mu_n}_{p} < m_n \text{ or } \nu_{n,p} > v_n} < + \infty \;.    
\end{equation}
In the next sections, we define the test supermartingales and provide power guarantees under $\cH_1$ or $\cH_{2,p}$. Note that these hypotheses include the i.i.d. case if $m_n$ and $v_n$ are constant but also include more general cases where the mean is allowed to vanish. As we will see in \Cref{sec:explicit-bounds-d}, the first-order hypothesis $\cH_1$  restricting the first-order moments is well suited for the Hoeffding test supermartingale while the second-order hypothesis $\cH_{2,p}$ restricting the second-order moments as well is tailored for the Capital test supermartingale for which we can use second-order betting strategies such as Online Newton Steps (ONS). 

\subsection{Definition of the test supermartingales}\label{sec:supermartingales-definitions}
In this section, we assume that \Cref{ass:bounded-y} holds and introduce the Hoeffding and Capital test supermartingales studied in this work.

\subsubsection{Hoeffding test supermartingale}\label{sec:hoeffding}

For all $\lambda\in\rset^d$ and betting strategy $(\lambda_n)_{n\geq 1} \subset \rset^d$, define the Hoeffding test supermartingale and its predictable plug-in counterpart as
\begin{equation}\label{eq:hoeffding-LW}
L_n^{\rm H}(\lambda) = \prod_{t=1}^n \exp\lr{\lambda^\top X_t - \norm{\lambda}_2^2 D^2/8} \quad \text{and} \quad
W_n^{\rm H} = \prod_{t=1}^n \exp\lr{\lambda_t^\top X_t - \norm{\lambda_t}_2^2 D^2/8}\;,\quad n \in \nset \,.   
\end{equation}
Then the following proposition holds.
\begin{proposition}\label{prop:hoeffding-martingales}
For any $\lambda \in \rset^d$ and betting strategy $(\lambda_n)_{n\geq 1}\subset \rset^d$, $(L_n^{\rm H}(\lambda))_{n\in \nset}$ and $(W_n^{\rm H})_{n\in\nset}$ are test supermartingales for $\cH_0$ of \eqref{eq:null-X}. 
\end{proposition}
\begin{proof}
    As discussed in \Cref{sec:plugin}, we only have to prove that, under 
    $\cH_0$, $(L_n^{\rm H}(\lambda))_{n \in \nset}$ is a  supermartingale for any $\lambda \in \Lambda$. This is true because, by Hoeffding's lemma, $\PEarg[n-1]{\exp\lr{\lambda^\top X_n - \norm{\lambda}_2^2 D^2/8}} \leq \rme^{\lambda^\top \PEarg[n-1]{X_n}} =1$ under $\cH_0$.
\end{proof}

\subsubsection{Capital test supermartingale}\label{sec:capital}
Let $\Gamma \subset \cB_{1/(2B)}^d$. Then for any $\gamma \in \Gamma$ and betting strategy $(\gamma_n)_{n\geq 1} \subset \Gamma$, define the Capital test supermartingale and its predictable plug-in counterpart as 
\begin{equation}\label{eq:capital-LW}
L_n^{\rm C}(\gamma) = \prod_{t=1}^n \lr{1 + \gamma^\top X_t}
\quad \text{and}  \quad
W_n^{\rm C} = \prod_{t=1}^n \lr{1 + \gamma_t^\top X_t}\;, \quad n \geq 1\,.
\end{equation}
Then the following proposition holds. 
\begin{proposition}\label{prop:capital-martingales}
For all $\gamma \in \Gamma$, and betting strategy $(\gamma_n)_{n\geq 1} \subset \Gamma$, the processes  $(L_n^{\rm C}(\lambda))_{n\in \nset}$ and $(W_n^{\rm C})_{n\in\nset}$ are test martingales for $\cH_0$ of \eqref{eq:null-X}. 
\end{proposition}
\begin{proof}
    This is true because, under $\cH_0$, $\PEarg[n-1]{1 + \gamma_t^\top X_t} = 1 + \gamma_t^\top \PEarg[n-1]{X_t} = 1$. 
\end{proof}

\subsubsection{Two steps capital test supermartingale}\label{sec:capital-2steps}
In the next sections, we will also study the power of the Capital test supermartingale introduced in \cite[Section~3]{Shekhar21} and which consists in defining the betting strategy $(\gamma_n)_{n\geq 1}$ of \eqref{eq:hoeffding-LW}  using a two steps approach. In the first step, we try to find the direction with the largest projection for $X_t$ and, in the second step, we chose the right bet along this direction.  Formally, 
for $\gamma \in [-1/2,1/2]$ and two predictable processes $(\gamma_n)_{n\geq 1} \subset [-1/2,1/2]$ and $(\eta_n)_{n\geq 1} \subset \cB_{1/B}^d$, define
\begin{equation}\label{eq:capital-2steps}
L_n^{\rm C, 2steps}(\gamma) = \prod_{t=1}^n (1 + \gamma \eta_t^\top X_t) \quad \text{and} \quad  W_n^{\rm C, 2steps} = \prod_{t=1}^n (1 + \gamma_t 
\eta_t^\top X_t) \,, \quad n \in \nset    \;,
\end{equation}
which are clearly test supermartingales for $\cH_0$ of \eqref{eq:null-X} similarly to \Cref{prop:capital-martingales}.
\subsection{Limiting cases and lower bounds}\label{sec:lower-bounds}
 We start by providing limit cases for the vanishing rate of $m_n$ where finite rejection time cannot be guaranteed and provide lower bounds for the expected rejection time when $m_n$ does not exceed a given threshold. 

\subsubsection{Hoeffding test supermartingale}
We start with the Hoeffding test supermartingale of \Cref{sec:hoeffding} and define, for all $\alpha \in (0,1)$, the rejection time at level $\alpha$ by 
$\tau_\alpha^{\rm H} := \inf\set{n\in\nset}{W_n^{\rm H} \geq 1/\alpha}$.
We rely on the following non-restrictive assumption on the betting strategy.
\begin{assumption}\label{ass:regret-hoeffding}
    For any process $(X_t)_{t\in \nset}$, the betting strategy $(\lambda_t)_{t\geq 1}$ constructed using $(X_t)_{t \in \nset}$ satisfies $\inf_{n \geq 1} \cR_n \geq 0$, where $\cR_n := \max_{\lambda\in\rset^d} \log L_n^{\rm H}(\lambda) - \log W_n^{\rm H}$. 
\end{assumption}  

Our first result shows that, under $\cH_1$, we cannot reject the null if $m_n$ vanishes faster than $\bigo{1/\sqrt{n}}$.

\begin{proposition}\label{prop:limit-case-hoeffding}
Assume that \Cref{ass:regret-hoeffding} holds. Then the following assertions hold.
\begin{enumerate}
    \item\label{itm:limit-case-hoeffding-1} For all $\alpha \in (0,1)$, there exist $m > 0$ and a process $(X_t)_{t \in \nset}$  which satisfies $\norm{\mu_n}_2 \geq m / \sqrt{n}$ for all $n \geq 1$ and  $\PP[\tau_\alpha^{\rm H} = +\infty] = 1$. 
    \item\label{itm:limit-case-hoeffding-2} For any deterministic sequence $(m_n)_{n \geq 1}$ such that $m_n = \smallo{1/\sqrt{n}}$, there exists a process $(X_t)_{t \in \nset}$ which satisfies $\norm{\mu_n}_2 \geq m_n$ for all $n \geq 1$ and such that  $\PP[\tau_\alpha^{\rm H} = +\infty] = 1$ for all $\alpha \in (0,1)$.
\end{enumerate}
\end{proposition}

Our second result provides a lower bound on the rejection time under $\cH_1$ when $m_n$ does not exceed an upper bound $m > 0$. 

\begin{proposition}\label{prop:hoeff-lower-power}
Assume that \Cref{ass:regret-hoeffding} holds. Then for all $m > 0$, there exists a process $(X_t)_{t\in \nset}$ which satisfies $\norm{\mu_n}_2 \geq m$ for all $n \geq 1$ and such that for all $\alpha \in (0,1)$,
$\PP[\tau_{\alpha}^{\rm H} \geq \frac{D^2\log(1/\alpha)}{2m^2}] = 1$.
\end{proposition}

\subsubsection{Capital test supermartingale}
We now derive similar results for the Capital test supermartingale of \Cref{sec:capital}. Define, for all $\alpha \in (0,1)$, the rejection time at level $\alpha$ by $
\tau_\alpha^{\rm C} := \inf\set{n\in\nset}{W_n^{\rm C} \geq 1/\alpha}$.
Our first result shows that, under $\cH_{2,p}$, we cannot reject the null if $m_n$ vanishes faster than $\bigo{1/n}$.
\begin{proposition}\label{prop:limit-case-capital}
The following assertions hold.
\begin{enumerate}
    \item\label{itm:limit-case-capital-1} For all $\alpha \in (0,1)$, there exist $m > 0$ and a process $(X_t)_{t \in \nset}$  which satisfies $\norm{\mu_n}_\infty \geq m / n$ for all $n \geq 1$ and  $\PP[\tau_\alpha^{\rm C} = +\infty] = 1$. 
    \item\label{itm:limit-case-capital-2} For any deterministic sequence $(m_n)_{n \geq 1}$ such that $m_n = \smallo{1/n}$, there exists a process $(X_t)_{t \in \nset}$ which satisfies $\norm{\mu_n}_\infty \geq m_n$ for all $n \geq 1$ and such that  $\PP[\tau_\alpha^{\rm C} = +\infty] = 1$ for all $\alpha \in (0,1)$.
\end{enumerate}
\end{proposition}

Our second result provides a lower bound on the rejection time under $\cH_1$ when $m_n$ does not exceed an upper bound $m > 0$. 

\begin{proposition}\label{prop:capital-lower-power}
For all $m > 0$, there exists a process $(X_t)_{t\in \nset}$ which satisfies $\norm{\mu_n}_\infty \geq m$ for all $n \geq 1$ and such that for all $\alpha \in (0,1)$,
$
\PP[\tau_{\alpha}^{\rm C} \geq \frac{2B\log(1/\alpha)}{m}] = 1
$.
\end{proposition}

\subsection{General power guarantees}\label{sec:power-general}
In this section, we study the power of the Hoeffding and Capital test supermartingales in a general form under $\cH_1$ and $\cH_{2,p}$. Then, we derive deterministic lower bounds for the Hoeffding and Capital test supermartingales which, as a consequence of \Cref{lem:power_from_bound},  immediately provide general power guarantees when the vanishing rate of $m_n$ is controlled. The next section is dedicated to particular cases where explicit power bounds can be computed.

\subsubsection{Hoeffding test supermartingale}
We start with the Hoeffding test supermartingale of \Cref{sec:hoeffding} and provide a deterministic lower bound for $\log W_n^{\rm H}$ and general power guarantees.
\begin{theorem}\label{thm:power-hoeff-nodim}
Assume that the regret $\cR_n := \max_{\lambda\in\rset^d} \log L_n^{\rm H}(\lambda) - \log W_n^{\rm H}$ of the betting strategy $(\lambda_n)_{n\geq 1}$ satisfies 
$\rho := \sum_{n \geq 1} \PP[\cR_n > r_n] < +\infty$ for some nonnegative sequence $(r_n)_{n\geq 1}$. 
Then, under the alternative $\cH_1$ defined in \eqref{eq:H1}, $(W_n^{\rm H})_{n\geq 1}$ satisfies \eqref{eq:boundW} for any $\alpha \in (0,1)$ with $\varrho = \rho + \varrho_1 + \frac{\pi^2}{3}$ and 
\begin{equation}\label{eq:un-hoeff}
u_n := \frac{2n\lr{m_n - 2D\sqrt{\log(n)/n}}_+^2}{D^2}- r_n \;.    
\end{equation}
Hence, we have $\liminf_{n\to +\infty} \frac{\log W_n^{\rm H}}{u_n} \geq 1  \;\;\PP\text{-}a.s$ and $\PE[\tau_\alpha^{\rm H}] \leq \rho + \varrho_1 + \frac{\pi^2}{3} + \aleph((u_n)_{n\geq 1},\log\lr{1/\alpha})$. 
\end{theorem}

\subsubsection{Capital test supermartingale}
We now derive similar results for the Capital test supermartingale of \Cref{sec:capital}. 
We let $(\rme_1,\cdots,\rme_{2d})$ be such that $(\rme_1,\cdots,\rme_d)$ is the canonical basis of $\rset^d$ and $\rme_{d+i} = -\rme_i$ for all $i=1,\cdots,d$.  

\begin{theorem}\label{thm:capital-power}
Assume that the regret $\cR_n := \max_{\gamma\in\Gamma} \log L_n^{\rm C}(\gamma) - \log W_n^{\rm C}$ of the betting strategy $(\gamma_n)_{n\geq 1}$
 satisfies
$\rho := \sum_{n \geq 1} \PP[\cR_n > r_n] < +\infty$
for some nonnegative sequence $(r_n)_{n\geq 1}$. Then, under the alternative $\cH_{2,\infty}$ defined in \eqref{eq:H2}, $(W_n^{\rm C})_{n\geq 1}$ satisfies \eqref{eq:boundW} for any $\alpha \in (0,1)$ with $\varrho = \rho + \varrho_{2,\infty} + \frac{\pi^2}{6}$ and
$$
u_n :=  n \epsilon_n m_n - 4 n \epsilon_n^2 v_n - r_n - 2\log(2dn^2)\;,
$$
for any deterministic sequence $(\epsilon_n)_{n\geq 1} \subset \cE$ with $\cE = \set{\epsilon > 0}{\forall i=1,\dots,2d, \epsilon \rme_i \in \Gamma}$. In particular we can take $(u_n)_{n\geq 1}$ as follows.
\begin{enumerate}
    \item\label{itm:capital-power-aggregation} If $\lrc{\epsilon \rme_1,\cdots,\epsilon\rme_{2d}}\subset\Gamma$ for some fixed $\epsilon \in (0,\frac{1}{2B}]$, then
\begin{equation}\label{eq:un-capital-aggregation}
u_n := \epsilon nm_n - 4\epsilon^2 nv_n - 2\log(2dn^2) - r_n \;.
\end{equation}
\item\label{itm:capital-power} If $\set{\epsilon \rme_i}{\epsilon\in[0,\frac{1}{2B}], i=1,\dots,2d}\subset \Gamma$, then 
\begin{equation}\label{eq:un-capital}
u_n := \frac{n m_n}{4}\lr{\frac{1}{B} \wedge \frac{m_n}{4v_n} } - 2\log(2dn^2) - r_n   \;.
\end{equation}
\end{enumerate}
Hence $\liminf_{n\to +\infty} \frac{\log W_n^{\rm C}}{u_n} \geq 1  \;\;\PP\text{-}a.s$ and $\PE[\tau_\alpha^{\rm C}] \leq \rho + \varrho_{2,\infty} + \frac{\pi^2}{6} + \aleph((u_n)_{n\geq 1},\log\lr{1/\alpha})$. 
\end{theorem}

\subsubsection{Two steps capital test supermartingale}
To conclude this section, we study the Capital 2steps strategy of \Cref{sec:capital-2steps} and provide a deterministic lower bound on $\log W_n^{\rm C, 2steps}$. 

\begin{theorem}\label{thm:capital-2step-bound}
Assume that the regret $\cR_n := \max_{\gamma \in [-1/2,1/2]} \log L_n^{\rm C}(\gamma) - \log W_n^{\rm C}$ of the betting strategy $(\gamma_n)_{n\geq 1}$ and the stochastic regret $\cS_n := \sup_{\eta \in \cB_{1/B}^d} \sum_{t=1}^n \PEarg[t-1]{\eta^\top X_t} - \sum_{t=1}^n \PEarg[t-1]{\eta_t^\top X_t}$ of $(\eta_n)_{n\geq 1}$ respectivly satisfy 
$\rho := \sum_{n \geq 1} \PP[\cR_n > r_n] < +\infty $ and 
$\varsigma := \sum_{n \geq 1} \PP[\cS_n > s_n] < +\infty
$,
for some nonnegative sequences $(r_n)_{n\geq 1}$ and $(s_n)_{n\geq 1}$. 
Then, under the alternative $\cH_{2,2}$ of \eqref{eq:H2},  $(W_n^{\rm C, 2steps})_{n\geq 1}$ satisfies \eqref{eq:boundW} for any $\alpha \in (0,1)$ with $\varrho = \rho + \varrho_{2,2} +\varsigma + \frac{\pi^2}{3}$ and
$$
u_n := \frac{(nm_n-s_n)_+}{4}\lr{1 \wedge \frac{(nm_n-s_n)_+}{4n v_n} } - 4\log(n) - r_n \;.
$$
Hence $\liminf_{n\to +\infty} \frac{\log W_n^{\rm C, 2steps}}{u_n} \geq 1  \;\;\PP\text{-}a.s$ and $\PE[\tau_\alpha^{\rm C, 2steps}] \leq \rho + \varrho_{2,2} + \varsigma + \frac{\pi^2}{3} + \aleph((u_n)_{n\geq 1},\log\lr{1/\alpha})$. 
\end{theorem}

\subsubsection{Discussion on the bounds and the impact of the dimension $d$}\label{sec:discussion-dimension}

To the best of our knowledge, the best regret bounds for the betting strategies used in the Hoeffding and Capital supermartingales are logarithmic, i.e. $r_n  = \bigo{\log(n)}$. Additionally, the stochastic regret for the projection step in the 2 steps Capital supermartingale can achieve $s_n = \bigo{\sqrt{n\log(n)}}$. We provide details in \Cref{sec:explicit-bounds-d}. With this in mind, we observe that \Cref{thm:power-hoeff-nodim,thm:capital-power,thm:capital-2step-bound} provide power guarantees at different order of generality depending on the size of the alternative. Namely, 
\Cref{thm:power-hoeff-nodim} and \Cref{thm:capital-2step-bound} apply only if $m_n$ is at least $\bigo{\sqrt{\log(n)/n}}$ which is not necessary for \Cref{thm:capital-power}. Similarly, Assertion~\ref{itm:capital-power-aggregation} in \Cref{thm:capital-power} applies only if $v_n$ is at most $\bigo{m_n}$ and $m_n$ needs to be at least $\bigo{\log(n)/n}$ while for Assertion~\ref{itm:capital-power},  $v_n$ can  dominate $m_n$ if $m_n^2/v_n$ is at least $\bigo{\log(n)/n}$. These rates are near-optimal compared to the $\bigo{1/\sqrt{n}}$ and $\bigo{1/n}$  limit vanishing mean rates for the Hoeffding and Capital test supermartingales respectively as shown in \Cref{prop:limit-case-hoeffding,prop:limit-case-capital}. ~\\

These vanishing rates for $m_n$ are comparable to the case studied in \cite[Theorem~2]{Shekhar21} where the authors show, in particular that, for an i.i.d. sequence $(X_t)_{t \in \nset}$ the Capital 2steps strategy of \Cref{sec:capital-2steps} has a detection boundary of $\bigo{\sqrt{\log(n)/n}}$ in the sense that for all $n \geq 1$, $\PP[\tau_\alpha^{\rm C, 2steps} > n]$ is controlled under the alternative $\PE[X_0] \geq m_n$ with $m_n = \bigo{\sqrt{\log(n)/n}}$. 
Our results tend to believe that, the Capital test martingale of \Cref{sec:capital} would achieve a detection boundary of order $\bigo{\log(n)/n}$ due to better second-order moment properties under additionnal variance contraints in the alternative. ~\\

While \Cref{thm:power-hoeff-nodim,thm:capital-2step-bound} are the more restrictive for $m_n$, they have the advantage of providing dimension free bounds and the ability to consider an alternative on the euclidean norm. On the other side, \Cref{thm:capital-power} considers an alternative on the infinite norm and provides a dimension-dependent bound. Since $\norm{x}_{\infty} \leq \norm{x}_2 \leq \sqrt{d} \norm{x}_{\infty}$, the alternative $\cH_{2,\infty}$ is more restrictive than $\cH_{2,2}$. To apply \Cref{thm:capital-power} for an alternative in euclidean norm, we can use the fact that $\cH_{2,\infty}$ is implied by the alternative $\cH_{2,2}' : \sum_{n \geq 1} \PParg{\norm{\mu_n}_{2} < \sqrt{d} m_n \text{ or } \nu_{n,2} > v_n} < + \infty$,
which adds another dependence on the dimension in the bound. All in all the dimension deteriorates the 1 step Capital test supermartingales performances whereas the Hoeffding and 2 steps Capital test supermartingale are much more robust to the dimension. 
In the next section, we specify the betting strategies used the rates $(m_n)_{n\geq 1}$ and $(v_n)_{n\geq 1}$ in the alternatives and provide explicit power bounds. 

\subsection{Explicit power bounds}\label{sec:explicit-bounds-d}
In this section, we provide examples of alternatives where the bounds obtained using \Cref{lem:power_from_bound} and \Cref{thm:power-hoeff-nodim,thm:capital-power} can be computed. 

\subsubsection{Hoeffding test supermartingale}
We start by providing power guarantees for the test supermartingale $(W_n^{\rm H, FTL})_{n\geq 1}$ which we define as the Hoeffding test supermartingale of \eqref{eq:hoeffding-LW} with Follow The Leader (FTL) as the betting strategy. 

\begin{lemma}\label{lem:regret-ftl}
Define $(W_n^{\rm H, FTL})_{n\geq 1}$ as in \eqref{eq:hoeffding-LW} where, for all $n\in\nset$,
$$
\lambda_{n+1} = \argmax_{\lambda\in\rset^d}\log L_{n}^{\rm H}(\lambda) = \frac{4\hat\mu_{n}}{D^2} \;.
$$
Then for all $n \geq 1$,
$\displaystyle
\max_{\lambda\in\rset^d} \log L_n^{\rm H}(\lambda) - \log W_n^{\rm H, FTL} \leq 4(1 + \log(n))$. 
\end{lemma}

Then, the following result holds.

\begin{corollary}\label{cor:hoeffding-examples-multi}
Define $(W_n^{\rm H, FTL})_{n\geq 1}$ as in \Cref{lem:regret-ftl} and let $\tau_{\alpha}^{\rm H,FTL}$ be its rejection time at level $\alpha$. Assume that $\cH_1$ holds. Then the following assertions hold. 
\begin{enumerate}
    \item\label{itm:hoeffding-examples-multi-1} If $m_n = m n^{-a}$ for some $m>0$ and $0 \leq a < 1/2$,  then
    $\displaystyle
    \liminf_{n \to +\infty} \frac{\log W_n^{\rm H,FTL}}{n^{1-2a}} \geq \frac{2m^2}{D^2}$, $\PP\text{-a.s.}$,
    and
    $$
    \PEarg{\tau_\alpha^{\rm H, FTL}} \leq \bigo{\lr{\linlog\lr{\frac{D^2}{m^2(1-2a)}} + \frac{D^2\log(1/\alpha)}{m^2}}^{\frac{1}{1-2a}}} \;.
    $$
    \item\label{itm:hoeffding-examples-multi-2} If $m_n = m \sqrt{\log(n)/n}$ for some $m > (2 + \sqrt{2}) D$, then 
    $\displaystyle
    \liminf_{n \to +\infty} \frac{\log W_n^{\rm H,FTL}}{\log(n)} \geq \frac{2m(m-4D)+4D^2}{D^2}$, $\PP\text{-a.s.}$, 
    and
    $$
    \PEarg{\tau_\alpha^{\rm H, FTL}} \leq \bigo{\exp\lr{\frac{D^2\log(1/\alpha)}{m(m-4D)+2D^2}}} \;.
    $$
\end{enumerate}
\end{corollary}
       The upper-bound on the expectation of the rejection time explodes under the largest alternative, i.e. the smallest $m$. It is due to the FTL strategy that achieves optimal regret $\bigo{\log n}$. To improve the rate, one has to consider second-order test martingales such as Capital test supermartingale.

\subsubsection{Capital test supermartingale}
We now provide power guarantees for two Capital test supermartingales $(W_n^{\rm C, EWA})_{n\geq 1}$ and $(W_n^{\rm C, ONS})_{n\geq 1}$ define as the Capital test supermartingale of \eqref{eq:hoeffding-LW} with respectively Exponential Weighted Average (EWA) and Online Newton Steps (ONS) as the betting strategy. 
First, the EWA betting strategy achieves the following regret bound.  
\begin{lemma}\label{lem:regret-ewa}
Let $\epsilon \in [0, \frac{1}{2B}]$ and let $g_k = \epsilon \rme_k$ for $k=1,\cdots,2d$. Define $(W_n^{\rm C,EWA})_{n\geq 1}$ as in \eqref{eq:hoeffding-LW} where, for all $n\geq 1$,
$$
\gamma_n = \frac{\sum_{k=1}^{2d}L_{n-1}^{\rm C}(g_k) g_k}{\sum_{j=1}^{2d} L_{n-1}^{\rm C}(g_k)} \;.
$$
Then for all $n \geq 1$, 
$\displaystyle
\max_{k=1,\cdots,2d} \log L_n^{\rm C}\lr{g_k} - \log W_n^{\rm C,EWA} \leq \log(2d)$.
\end{lemma}

Then, the following result holds.

\begin{corollary}\label{cor:capital-examples-aggregation-multi}
Define $(W_n^{\rm C,EWA})_{n\geq 1}$ as in \Cref{lem:regret-ewa} for some $\epsilon \in (0, \frac{1}{2B}]$ and let $\tau_{\alpha}^{\rm C,EWA}$ be its rejection time at level $\alpha$. Assume that $\cH_2$ holds with $m_n = mn^{-a}$ for some $0 < a < 1$ and $m > 0$.  Then the following assertions hold. 
\begin{enumerate}
\item If $v_n = vn^{-a}$ and $\epsilon < \frac{m}{4v}$, then 
$\displaystyle
\liminf_{n \to +\infty} \frac{\log W_n^{\rm C,EWA}}{n^{1-a}} \geq \epsilon(m-4\epsilon v)$, $\PP\text{-a.s.}$,
and
$$
\PE[\tau_{\alpha}^{\rm C,EWA}] \leq \bigo{ \lr{\linlog\lr{\frac{1}{\epsilon(m-4\epsilon v)(1-a)}} + \frac{\log(d/\alpha)}{\epsilon(m-4\epsilon v)}}^{\frac{1}{1-a}}} \;.
$$
\item If $v_n = vn^{-2b}$ with $v>0$ and $a/2 < b < 1/2$, then 
$\displaystyle
\liminf_{n \to +\infty} \frac{\log W_n^{\rm C,EWA}}{n^{1-a}} \geq \epsilon m$, $\PP\text{-a.s.}$, 
and
$$
\PE[\tau_{\alpha}^{\rm C,EWA}] \leq \bigo{\lr{\frac{8\epsilon v}{m}}^{\frac{1}{2b-a}}} \vee \bigo{\lr{\linlog\lr{\frac{1}{\epsilon m(1-a)}} + \frac{\log(d/\alpha)}{\epsilon m}}^{\frac{1}{1-a}}}\;.
$$
\item If $v_n = vn^{-1}$ with $v>0$, then 
$\displaystyle
\liminf_{n \to +\infty} \frac{\log W_n^{\rm C,EWA}}{n^{1-a}} \geq \epsilon m$, $\PP\text{-a.s.}$, 
and
$$
\PE[\tau_{\alpha}^{\rm C,EWA}] \leq \bigo{\lr{\linlog\lr{\frac{1}{\epsilon m(1-a)}} + \frac{\log(d/\alpha) + \epsilon^2 v}{\epsilon m}}^{\frac{1}{1-a}}} \;.
$$
\item If $v_n = v \log(n)/n$ with $v>0$, then 
$\displaystyle
\liminf_{n \to +\infty} \frac{\log W_n^{\rm C,EWA}}{n^{1-a}} \geq \epsilon m$, $\PP\text{-a.s.}$, 
and
$$
\PE[\tau_{\alpha}^{\rm C,EWA}] \leq \bigo{\lr{\linlog\lr{\frac{1+\epsilon^2 v}{\epsilon m(1-a)}} + \frac{\log(d/\alpha)}{\epsilon m}}^{\frac{1}{1-a}}} \;.
$$
\end{enumerate}
\end{corollary}
It is remarkable to consider rates $n^{-a}$ with $a \ge 1/2$, beyond the one of the law of the iterated logarithm. It is possible under an alternative with a fast rate on the control of the variance. Such trick is possible thanks to the Capital test supermartingale which takes into account second-order properties. We recover the limit rate $m_n = m/n$ $(a=1)$ of \Cref{prop:limit-case-capital}.

Now, the ONS betting strategy allows to take a non-finite set $\Gamma$. The strategy is detailed in \Cref{alg:ons} and achieves the following regret.

\begin{algorithm}[h]
\caption{Online Newton Step for the Capital process}\label{alg:ons}
\begin{algorithmic}
\Require A subset $S$ of $\rset^d$. 
\State {\bf Initialize :} $\gamma_1 = 0$, $A_0 = I_d$ 
\For{$t \geq 1$}
    \State Observe $x_t \in \cB_{1}^d$
    \State Set $z_t = \frac{-x_t}{1+\gamma_t^\top x_t}$ and $A_t = A_{t-1} + z_t z_t^\top$
    \State Set $\gamma_{t+1} = \Pi_{S}^{A_t}\lr{\gamma_t - \frac{2}{2-\log(3)} A_t^{-1}z_t}$ where $\Pi_S^A(x) = \argmin_{y \in S} \pscal{A (y-x)}{y-x}$
\EndFor
\end{algorithmic}
\end{algorithm}

 \begin{lemma}\label{lem:regret-ons}
 Define $(W_n^{\rm C, ONS})_{n\geq 1}$ as in \eqref{eq:hoeffding-LW} where $(\gamma_n)_{n\geq 1} \subset \Gamma := \cB_{1/(2B)}^d$ is constructed using the ONS algorithm detailed in \Cref{alg:ons} with $S = \cB_{1/2}^d$ and $x_t = X_t/B$. Then for all $n \geq 1$,
 $$
\max_{\gamma \in \Gamma} \log L_n^{\rm C}(\gamma) - \log W_n^{\rm C, ONS} \leq d\lr{7.2 + 4.5\log(n)} \;.
$$ 
 The same regret bound applies with $d = 1$, $\Gamma = [0,1/(2B)]$ and $S = [0,1/2]$. 
\end{lemma}

Then, the following result holds.
\begin{corollary}\label{cor:capital-examples-multi}
Define $(W_n^{\rm C, ONS})_{n\geq 1}$ as in \Cref{lem:regret-ons} and  let $\tau_{\alpha}^{\rm C,ONS}$ be its rejection time at level $\alpha$. Assume that $\cH_{2,\infty}$ holds with $m_n = m n^{-a}$ for some $m > 0$ and $0\leq a < 1$. Then the following assertions hold.
\begin{enumerate}
    \item\label{itm:capital-examples-multi-1} If $v_n = vn^{-2b}$, $b\ge0$ and $a-1/2 < b \leq a/2$, then 
    $\displaystyle
\liminf_{n \to +\infty} \frac{\log W_n^{\rm C,ONS}}{n^{1-2(a-b)}} \geq \frac{m^2}{16v}$,  $\PP\text{-a.s.}$,
and 
\begin{equation}\label{eq:bound_tau_ONS}
\PE[\tau_\alpha^{\rm C,ONS}] \leq
 \bigo{\cA} \vee \lr{\bigo{\cB}\wedge\bigo{\cC}} \;,
\end{equation}
with
\begin{align*}
        \cA &= \lr{\linlog\lr{\frac{4vd}{m^2(1-2(a-b))}} + \frac{4v(d+\log(d/\alpha))}{m^2}}^{\frac{1}{1-2(a-b)}}     \\ 
    \cB &= \lr{\linlog\lr{\frac{Bd}{m(1-a)}} + \frac{B(d+\log(d/\alpha))}{m}}^{\frac1{1-a}}\\
    \cC &=  \lr{\frac{Bm}{4v}}^{\frac{1}{a-2b}} \;.
    \end{align*}
    \item If $v_n = vn^{-2b}$ and $b > a/2$, then 
    $\displaystyle
\liminf_{n \to +\infty} \frac{\log W_n^{\rm C,ONS}}{n^{1-a}} \geq \frac{m}{4B}$, $\PP\text{-a.s.}$,
and \eqref{eq:bound_tau_ONS} holds with
    \begin{align*}
        \cA &= \lr{\linlog\lr{\frac{Bd}{m(1-a)}} + \frac{B(d+\log(d/\alpha))}{m}}^{\frac1{1-a}}\\
    \cB &= \lr{\linlog\lr{\frac{4vd}{m^2(1-2(a-b))}} + \frac{4v(d+\log(d/\alpha))}{m^2}}^{\frac{1}{1-2(a-b)}}   \\
    \cC &=  \lr{\frac{4v}{Bm}}^{\frac{1}{2b-a}} \;.
    \end{align*}
    \item If $v_n = v\log(n)/n$, then 
    $\displaystyle
\liminf_{n \to +\infty} \frac{ \log W_n^{\rm C,ONS}}{n^{2(1-a)}/\log(n)} \geq \frac{m^2}{16v}$, $\PP\text{-a.s.}$,
and \eqref{eq:bound_tau_ONS} holds with
\begin{align*}
        \cA &=   \lr{\linlog\lr{\frac{Bd}{m(1-a)}} + \frac{B(d+\log(d/\alpha))}{m}}^{\frac1{1-a}}\\
    \cB &= \lr{\linlog\lr{\frac{\sqrt{4v(d+\log(d/\alpha))}}{m(1-a)}}}^{\frac1{1-a}}
 \\
    \cC &=  \lr{\linlog\lr{\frac{4v}{Bm(1-a)}}}^{\frac{1}{1-a}} \;.
    \end{align*}
    \end{enumerate}
\end{corollary}

In the case 2. one can let $b\to \infty$ and reach the degenerate setting of deterministic sequences. Then the upper bound is driven by $\cA$ and can still explode and we recover the limit rate $m_n = m/n$ $(a=1)$ of \Cref{prop:limit-case-capital}.

\subsubsection{Two steps capital test supermartingale}
Let us end the explicit power guarantees with the test supermartingale $(W_n^{\rm C, 2steps})_{n\geq 1}$ of \eqref{eq:capital-2steps}. For this test supermartingale, we use the ONS algorithm to select the size of the bets and the Online Gradient Ascent (OGA) algorithm to select the direction of the bets. The OGA algorithm achieves to following stochastic regret.

\begin{lemma}\label{lem:oga-stoch-regret}
Let $(\eta_n)_{n\geq 1}$ be constructed with the online projected gradient ascent (OGA) algorithm with gradient steps $\frac{2}{B^2 \sqrt{t}}$, that is 
$$
\eta_{t+1} = \Pi_{\cB_{1/B}^d}\lr{\eta_{t} + \frac{2 X_{t}}{B^2\sqrt{t}}}\;.
$$
Then for all $n\geq 1$, with probability at least $1-1/n^2$, we have
$$
\sup_{\eta \in \cB_{1/(2B)}^d} \sum_{t=1}^n \PEarg[t-1]{\eta^\top X_t} - \sum_{t=1}^n \PEarg[t-1]{\eta_t^\top X_t} \leq \sqrt{n} (1 + 4\sqrt{\log(n)}) \;.
$$
\end{lemma}

Then, the following result holds.
\begin{corollary}\label{cor:capital-examples-2step}
Consider the 2 steps test supermatingale of \eqref{eq:capital-2steps} and assume that $(\eta_n)_{n\geq 1}$ and $(\gamma_n)_{n\geq 1}$  are respectively constructed using the OGA algorithm and the ONS algorithm. Let $\tau_{\alpha}^{\rm C, 2steps}$ be its rejection time at level $\alpha$. Assume that $\cH_{2,2}$ holds with with $m_n = m n^{-a}$ for some $m > 0$ and $0\leq a < 1/2$. Then $\log W_n^{\rm C, 2steps}$ has the same asymptotical behavior as the ones obtained in \Cref{cor:capital-examples-multi} for $\log W_n^{\rm C, ONS}$ where we take $B = 1$. Moreover we have 
 $$
\PE[\tau_\alpha^{\rm C, 2steps}] \leq 
\bigo{\lr{\linlog\lr{\frac{1}{m^2(1-2a)}}}^{\frac{1}{1-2a}}} \vee \bigo{\cA} \vee \lr{\bigo{\cB}\wedge\bigo{\cC}} \;,
$$
where the expressions of $\cA,\cB,\cC$ depend on the range of $b$  as in \Cref{cor:capital-examples-multi} with $B = d =1$.
\end{corollary}
The upper-bound does not depend on the dimension $d$ at the price of the restriction $a<1/2$ on the alternative due to the regret's rate of OGA.
\subsubsection{Comparison of the bounds}
The bounds obtained for $\PE[\tau_{\alpha}^{\rm H,FTL}]$ are valid universally, for any bounded real-valued process. While the bounds obtained in for $\PE[\tau_{\alpha}^{\rm C,EWA}]$, $\PE[\tau_{\alpha}^{\rm C,ONS}]$ and $\PE[\tau_{\alpha}^{\rm C, 2steps}]$ are valid under some second moment assumptions. If no information is available on the second moment, we can always take $v_n = B^2$ since $\PEarg[t-1]{\norm{X_t}_\infty^2} \leq B^2$. In this case, \Cref{cor:capital-examples-2step}  recovers similar power guarantees as \Cref{cor:hoeffding-examples-multi} and so does Assertion~\ref{itm:capital-examples-multi-1} of \Cref{cor:capital-examples-multi} with an additional $\bigo{d\log(d)}$ dependence on the dimension.  \Cref{cor:capital-examples-aggregation-multi}, on the other hand, only applies when the second moment decreases at least as fast as the mean. 
~\\

It seems that the best choice between the three test martingales $W_n^{\rm H, FTL}, W_n^{\rm C, EWA}$, $W_n^{\rm C, ONS}$ and $W_n^{\rm C, 2steps}$ depends on a compromise between the size of the alternative and the dependence on the dimension. As seen in \Cref{cor:hoeffding-examples-multi,cor:capital-examples-aggregation-multi,cor:capital-examples-multi,cor:capital-examples-2step}, it seems that covering larger alternatives come at the cost of larger dependence on the dimension: while $\PE[\tau_{\alpha}^{\rm H,FTL}]$ and $\PE[\tau_{\alpha}^{\rm H,2steps}]$ are independent of $d$ but a restricted to $a < 1/2$, we get $\PE[\tau_{\alpha}^{\rm C,EWA}] \leq \bigo{\log(d)}$ and $\PE[\tau_{\alpha}^{\rm C,ONS}] \leq \bigo{d\log(d)}$ but are valid for $a \geq 1/2$. 
~\\

An interesting common alternative is the stationary one with $\PEarg[t-1]{X_t} = \PE[X_0]$ and $\PEarg[t-1]{\norm{X_t}_\infty^2} = \PE[\norm{X_0}_\infty^2]$ and where we take constant $v_n = v$ and $m_n = m$. In this case, we get 
\begin{align*}
\PE[\tau_{\alpha}^{\rm H, FTL}] &\leq \bigo{\linlog\lr{\frac{D^2}{m^2}}+\frac{D^2\log(1/\alpha)}{m^2}} \\
\PE[\tau_{\alpha}^{\rm C, EWA}] &\leq \bigo{\linlog\lr{\frac{B^2}{(Bm-2v)_+}}+\frac{B^2\log(d/\alpha)}{(Bm-2v)_+}} \\
\PE[\tau_{\alpha}^{\rm C, ONS}] &\leq \bigo{\linlog\lr{\frac{(4v \vee Bm) d}{m^2}}+\frac{(4v \vee Bm)(d+\log(d/\alpha)}{m^2}} \\
\PE[\tau_{\alpha}^{\rm C, 2steps}] &\leq \bigo{\linlog\lr{\frac{8v \vee m}{m^2}}+\frac{(8v \vee m)\log(1/\alpha)}{m^2}}\;.
\end{align*}
Noting that $v = \PE[\norm{X_0}_\infty^2] \geq \norm{\PE[X_0]}_\infty^2 \geq m^2$, we see that the bound of $\PE[\tau_{\alpha}^{\rm C, EWA}]$ is limited to $m \leq B/2$ and that the bound on $\PE[\tau_{\alpha}^{\rm C, ONS}]$ is  $\bigo{\linlog\lr{\frac{vd}{m^2}}+\frac{v(d+\log(d/\alpha))}{m^2}}$ for $m\geq B/4$. 
We recover the rates obtain in Section~3 of \cite{Shekhar21}, our second order term $v$ being looser than their variance term. However, our results are near-optimal compared to the lower rejection time bound obtained in \Cref{prop:capital-lower-power}.

\section{Extensions}\label{sec:extensions}
\subsection{Extension to a composite null}\label{sec:ineq_hyp}
In this section, we consider the one dimensional case ($d=1$) and still assume that \Cref{ass:bounded-y} holds. In this case all norms are equal the absolute value so we omit the subscript $p$ in $\nu_{n,p}$. We consider the composite null hypothesis
\begin{equation}\label{eq:null-Y-neg}
\cH_0^{-} \;:\; \PEarg[t-1]{X_t} \leq 0\,, \; \PP\text{-a.s.} \, \text{ for all } t\in\nset \;.
\end{equation}
Then, restricting the bets to nonnegative values, the Hoeffding and Capital processes remain test supermartingales for $\cH_0$ of \eqref{eq:null-Y-neg}. Note that the limiting cases and lower bounds obtained in \Cref{prop:limit-case-hoeffding,prop:hoeff-lower-power,prop:limit-case-capital,prop:capital-lower-power} remain valid. 
\begin{proposition}\label{prop:composite-martingales}
For all $\lambda \geq 0$, the process $(L_n^{\rm H}(\lambda))_{n\in\nset}$ defined in \eqref{eq:hoeffding-LW} is a test supermartingale for $\cH_0^{-}$ of \eqref{eq:null-Y-neg} and so is $(W_n^{\rm H})_{n\in\nset}$ defined in \eqref{eq:hoeffding-LW} if $\lambda_n \geq 0$ for all $n \geq 1$. In addition, for all $\gamma \in [0,1/(2B)]$, the process $(L_n^{\rm C}(\gamma))_{n\in\nset}$ defined in \eqref{eq:hoeffding-LW} is a test supermartingale for $\cH_0^{-}$ of \eqref{eq:null-Y-neg} and so is $(W_n^{\rm C})_{n\in\nset}$ defined in \eqref{eq:hoeffding-LW} if $\gamma_n \in [0,1/(2B)]$ for all $n \geq 1$.
\end{proposition}
\begin{proof}
    The statement about $(L_n(\lambda)^{\rm H})_{n\in \nset}$ comes from Hoeffding's lemma using the fact that, for any $\lambda \geq 0$, $\PEarg[n-1]{\exp(\lambda X_n - \lambda^2 D^2/8)} \leq \rme^{\lambda \PEarg[n-1]{X_n}} \leq 1$ under $\cH_0^{-}$. The statement about $(L_n(\gamma)^{\rm C})_{n\in \nset}$ comes from the fact that, for any $\gamma \in [0,1/(2B)]$, $\PEarg[n-1]{1+\gamma X_n} \leq 1$. 
\end{proof}
We therefore assume that $L_n^{\rm H}(\lambda)$ and $W_n^{\rm H}$ are respectively defined by \eqref{eq:hoeffding-LW} and \eqref{eq:hoeffding-LW} for $\lambda \geq 0$ and a betting strategy $(\lambda_n)_{n \geq 1} \subset \Lambda =\rset_+$.

\begin{theorem}\label{thm:power-hoeff-positif}
Assume that the regret $\cR_n := \max_{\lambda\geq 0} \log L_n^{\rm H}(\lambda) - \log W_n^{\rm H}$ of the betting strategy $(\lambda_n)_{n\geq 1}$ satisfies $
\rho := \sum_{n \geq 1} \PP[\cR_n > r_n] < +\infty$, for some nonnegative sequence $(r_n)_{n\geq 1}$. Let $(m_n)_{n \geq 1}$ be a nonnegative sequence.  Then, under the alternative 
$$
\cH_1 : \varrho_1:=\sum_{n\geq 1} \PP[\mu_n < m_n] < +\infty\;,
$$
the test supermartingale
$(W_n^{\rm H})_{n\geq 1}$ satisfies \eqref{eq:boundW}, for any $\alpha \in (0,1)$ with $\varrho = \rho + \varrho_1 + \frac{\pi^2}{6}$ and
$$
u_n := \frac{2n\lr{m_n - D\sqrt{\log(n)/n}}_+^2}{D^2}- r_n \;.
$$
Hence $\liminf_{n\to +\infty} \frac{\log W_n^{\rm H}}{u_n} \geq 1  \;\;\PP\text{-}a.s$ and $\PE[\tau_\alpha^{\rm H}] \leq \rho + \varrho_{1} + \frac{\pi^2}{6} + \aleph((u_n)_{n\geq 1},\log\lr{1/\alpha})$. 
\end{theorem}
Note that, unlike the case where $\Lambda = \rset$, we were not able to show that the betting strategy FTL achieves a logarithmic regret when $\Lambda = \rset_+$.  

As for the Hoeffding supermartingale, we assume now that $L_n^{\rm C}(\gamma)$ and $W_n^{\rm C}$ respectively defined by \eqref{eq:hoeffding-LW} and \eqref{eq:hoeffding-LW} for $\gamma \in [0,1/(2B)]$ and a betting strategy $(\gamma_n)_{n \geq 1} \subset \Gamma =[0,1/(2B)]$. 
\begin{theorem}\label{thm:power-capital-positif}
Assume that the regret $\cR_n := \max_{\gamma \in [0,1/(2B)]} \log L_n^{\rm C}(\gamma) - \log W_n^{\rm C}$ of the betting strategy $(\gamma_n)_{n\geq 1}$
 satisfies $
\rho := \sum_{n \geq 1} \PP[\cR_n > r_n] < +\infty$, for some nonnegative sequence $(r_n)_{n\geq 1}$. Let $(m_n)_{n\geq 1}$ and $(v_n)_{n\geq 1}$ be two nonnegative sequences.
Then, under the alternative 
$$
\cH_2 : \varrho_2 := \sum_{n \geq 1} \PParg{\mu_n < m_n \text{ or } \nu_n > v_n} < + \infty\;,
$$
the test supermaringale $(W_n^{\rm C})_{n\geq 1}$ satisfies \eqref{eq:boundW}, for any $\alpha \in (0,1)$ with $\varrho = \rho + \varrho_1 + \frac{\pi^2}{3}$ and 
$$
u_n := \frac{n m_n}{4}\lr{\frac{1}{B} \wedge \frac{m_n}{4v_n} }   - 4\log(n) -r_n \;.
$$
Hence, we have $\liminf_{n\to +\infty} \frac{\log W_n^{\rm C}}{u_n} \geq 1  \;\;\PP\text{-}a.s$ and $\PE[\tau_\alpha^{\rm C}] \leq \rho + \varrho_{2} + \frac{\pi^2}{3} + \aleph((u_n)_{n\geq 1},\log\lr{1/\alpha})$. 
\end{theorem}
The restriction to positive bets do not deteriorate the power properties of the Capital test supermartingales that extends easily to composite null. 
\subsection{Extension to other functionals}\label{sec:supg}
In this section, we observe a sequence $(X_t)_{t\in\nset}$ valued in a set $\Xset$ and consider a set $\cG$ of functions from $\Xset$ to $[-1,1]$. We are interested in the null hypotheses
\begin{equation}\label{eq:null-functional}
\cH_0 : \PEarg[t-1]{g(X_t)} = 0 \text{ for all } t \geq 1 \text{ and } g \in\cG \;,    
\end{equation}
\begin{equation}\label{eq:null-functional-ineq}
\cH_0^{-} : \PEarg[t-1]{g(X_t)} \leq 0 \text{ for all } t \geq 1 \text{ and } g \in\cG \;.
\end{equation}
Following \cite{Shekhar21}, we consider a predictable sequence $(g_t)_{t\geq 1}$ valued in $\cG$ referred to as the \emph{prediction strategy} and denote its stochastic regret by
$$
\cS_n := \sup_{g \in \cG} \sum_{t=1}^n \PEarg[t-1]{g(X_t)} - \sum_{t=1}^n \PEarg[t-1]{g_t(X_t)}\;.
$$
Throughout this section, we assume that there exists a nonnegative sequence $(s_n)_{n\geq 1}$, such that 
$$
\varsigma := \sum_{n \geq 1} \PP[\cS_n > s_n] < +\infty  \;.
$$ 
\subsubsection{Hoeffding test supermartingale}
Given a set $\Lambda \subset \rset$, for $\lambda \in \Lambda$ and a $\Lambda$-valued betting strategy $(\lambda_n)_{n\geq 1}$, we define the Hoeffding test supermartingales as
\begin{equation}\label{eq:hoeffding-functional}
L_n^{\rm H}(\lambda) = \prod_{t=1}^n \exp\lr{\lambda g_t(X_t) - \lambda^2/2} \quad \text{and} \quad  W_n^{\rm H} = \prod_{t=1}^n \exp\lr{\lambda_t g_t(X_t) - \lambda_t^2/2} \,, \quad n \in \nset \;.
\end{equation}
The following proposition holds. 
\begin{proposition}\label{prop:hoeffding-martingales-functional}
Relation \eqref{eq:hoeffding-functional} defines two test supermartingales for $\cH_0$ of \eqref{eq:null-functional} if we take $\Lambda = \rset$ and for $\cH_0^{-}$ of \eqref{eq:null-functional-ineq} if we take $\Lambda = \rset_+$.
\end{proposition}
\begin{proof}
The proof is similar to the proof of \Cref{prop:hoeffding-martingales} using the fact that for any $\lambda \in \rset$, $\lambda \PEarg[n-1]{g_n(X_n)} \leq \abs{\lambda}\sup_{g\in\cG}\abs{\PEarg[n-1]{g(X_n)}} = 0$ under $\cH_0$ and for any $\lambda \geq 0$, $\lambda \PEarg[n-1]{g_n(X_n)} \leq \lambda\sup_{g\in\cG}\PEarg[n-1]{g(X_n)} \leq 0$ under $\cH_0^{-}$.  
\end{proof}

The following result extends \Cref{thm:power-hoeff-nodim,thm:power-hoeff-positif} to other functionals.
\begin{theorem}\label{thm:hoeffding-g-bound}
Let $\Lambda = \rset$ or $\rset_+$ and assume that the regret $\cR_n := \max_{\lambda\in\Lambda} \log L_n^{\rm H}(\lambda) - \log W_n^{\rm H}$ of the betting strategy $(\lambda_n)_{n\geq 1}$ satisfies 
$
\rho := \sum_{n \geq 1} \PP[\cR_n > r_n] < +\infty$,
for some nonnegative sequence $(r_n)_{n\geq 1}$. Let $(m_n)_{n \geq 1}$ be a nonnegative sequence and consider the alternative hypothesis 
$$
 \cH_1 : \varrho_1 :=\sum_{n\geq 1} \PParg{\sup_{g\in\cG}\frac{1}{n}\sum_{t=1}^n \PEarg[t-1]{g(X_t)}  < m_n} < +\infty \;,
$$
Then, under $\cH_1$,  $(W_n^{\rm H})_{n\geq 1}$ satisfies \eqref{eq:boundW} for any $\alpha \in (0,1)$ with $\varrho = \rho + \varsigma + \varrho_1 + \frac{\pi^2}{6}$ and 
 $$
 u_n := \frac{1}{2n}\lr{nm_n-s_n-2\sqrt{n\log(n)}}_+^2 - r_n \;.
 $$
 Hence, we have $\liminf_{n\to +\infty} \frac{\log W_n^{\rm H}}{u_n} \geq 1  \;\;\PP\text{-}a.s$ and $\PE[\tau_\alpha^{\rm H}] \leq \rho + \varrho_1 + \varsigma + \frac{\pi^2}{6} + \aleph((u_n)_{n\geq 1},\log\lr{1/\alpha})$. 
\end{theorem}

Note that, unlike the case where $\Lambda = \rset$, we were not able to show that the betting strategy achieves a logarithmic regret when $\Lambda = \rset_+$. 

\subsubsection{Capital test supermartingale}
Given a set $\Gamma \subset \rset$, for $\gamma \in \Gamma$ and a $\Gamma$-valued betting strategy $(\gamma_n)_{n\geq 1}$, we define the Capital test supermartingales as
\begin{equation}\label{eq:capital-functional}
L_n^{\rm C}(\gamma) = \prod_{t=1}^n (1 + \gamma g_t(X_t)) \quad \text{and} \quad  W_n^{\rm C} = \prod_{t=1}^n (1 + \gamma_t g_t(X_t)) \,, \quad n \in \nset    \;.
\end{equation}
The following proposition, whose proof of similar to the one of \Cref{prop:hoeffding-martingales-functional}, holds.
\begin{proposition}\label{prop:capital-martingales-functional}
Relation \eqref{eq:capital-functional} defines two test supermartingales for $\cH_0$ of \eqref{eq:null-functional} if we take $\Gamma = [-1/2,1/2]$ and for $\cH_0^{-}$ of \eqref{eq:null-functional-ineq} if we take $\Gamma = [0,1/2]$.
\end{proposition}
The following result extends \Cref{thm:capital-power,thm:power-capital-positif} to other functionals. 
\begin{theorem}\label{thm:capital-g-bound}
Let $\Gamma = [-1/2,1/2]$ or $[0,1/2]$ and assume that the regret $\cR_n := \max_{\gamma \in \Gamma} \log L_n^{\rm C}(\gamma) - \log W_n^{\rm C}$ of the betting strategy $(\gamma_n)_{n\geq 1}$ satisfies
$
\rho := \sum_{n \geq 1} \PP[\cR_n > r_n] < +\infty
$,
for some nonnegative sequence $(r_n)_{n\geq 1}$. Let $(m_n)_{n\geq 1}$ and $(v_n)_{n\geq 1}$ be two nonnegative sequences and consider the alternative hypothesis 
$$
\cH_2 : \varrho_2 := \sum_{n \geq 1} \PParg{\sup_{g\in\cG}\frac{1}{n}\sum_{t=1}^n \PEarg[t-1]{g(X_t)} < m_n \text{ or } \sup_{g\in\cG}\frac{1}{n}\sum_{t=1}^n \PEarg[t-1]{g(X_t)^2} > v_n} < + \infty \;.
$$
Then, under $\cH_2$, $(W_n^{\rm C})_{n\geq 1}$ satisfies \eqref{eq:boundW} for any $\alpha \in (0,1)$ with $\varrho = \rho + \varrho_2 + \varsigma  + \frac{\pi^2}{3}$ and
$$
u_n := \frac{(nm_n-s_n)_+}{4}\lr{1 \wedge \frac{(nm_n-s_n)_+}{4n v_n} } - 4\log(n) - r_n \;.
$$
Hence, we have $\liminf_{n\to +\infty} \frac{\log W_n^{\rm C}}{u_n} \geq 1  \;\;\PP\text{-}a.s$ and $\PE[\tau_\alpha^{\rm C}] \leq \rho + \varrho_2 + \varsigma + \frac{\pi^2}{3} + \aleph((u_n)_{n\geq 1},\log\lr{1/\alpha})$. 
\end{theorem}

We observe the same behavior as discussed in \Cref{sec:discussion-dimension}, namely that the Hoeffding test supermartingale is restricted to alternatives where $n m_n$ is at least $\bigo{\sqrt{n\log(n)}}$ even if $s_n$ is of a lower order of magnitude.  For the Capital test supermartingale, we can hope for larger alternatives but, unlike \Cref{thm:capital-power}, we are restricted to the ones for which $n m_n$ increases at least as fast as $s_n$. Hence the performance of the prediction strategy directly impacts the size of the alternative. 

\section{Applications}\label{sec:application}
\subsection{Testing for elicitable and identifiable forecasters}\label{sec:test-forecast}
In this section, we specify a null hypothesis for  the evaluation of a forecaster and propose test supermartingales. We observe an $(\cF_t)_{t\in\nset}$-adapted process $(Y_t)_{t\in\nset}$ valued in a measurable space $(\Yset,\Ysigma)$ and consider the problem of predicting a statistical quantity $\theta_t\in\Theta$ of the distribution of $Y_t$ given $\cF_{t-1}$ where $\Theta\subset\rset^d$ for some $d\geq 1$. We assume that at each time step $t$, an expert provides a predictable forecast $\hat\theta_t$ of $\theta_t$.
We consider two cases: the \emph{identifiable} one and the \emph{elicitable} one. These cases are studied in \cite{cassgrain22-anytime} under the assumption that $\theta_t$ is constant over time. In the identifiable case, we assume that $\theta_t$ satisfies the identifiability condition
\begin{equation}\label{eq:identifiable}
\PEarg[t-1]{m(\theta_t, Y_t)} = 0\; \text{ for all } t \geq 1\,,
\end{equation}
for some known function $m: \Theta\times\Yset \to \Xset\subset \rset^d$. Hence, if $\Xset$ is bounded, this reduces to bounded mean testing studied in \Cref{sec:mean-test}  with $X_t = m(\hat\theta_t,Y_t)$. 

In the elicitable case, we assume that $\theta_t$ satisfies the elicitability condition
\begin{equation}\label{eq:elicitable}
\theta_t \in \argmin_{\theta\in\Theta} \PEarg[t-1]{\ell(\theta,Y_t)}\; \text{ for all } t \geq 1\,,
\end{equation}
for some known loss function $\ell : \Theta\times\Yset \to \rset$.  Observing that this condition is equivalent to 
$$
\PEarg[t-1]{\ell(\theta_t,Y_t) - \ell(\theta,Y_t)}  \leq 0 \; \text{ for all } \theta \in \Theta \;,
$$
we get that, if $\Yset,\Theta$ and $\ell$ are bounded, then the elicitable case lies in the setting studied in \Cref{sec:supg} for the composite null taking $X_t = (\theta_t,X_t) \in \Xset = \Theta \times \Yset$ and $\cG = \set{(\theta,y) \mapsto \tilde\ell(\theta,y) -  \tilde\ell(\xi,y)}{\xi \in \Theta}$ where $\tilde\ell$ is a scaled version of $\ell$ so that functions in $\cG$ are valued in $[-1,1]$.

In \cite{cassgrain22-anytime}, the authors consider tests for elicitable and identifiable functionals via the null hypothesis defined by their Equation~(8). This context is similar to ours if we assume that $\theta_t = \theta_0$ is constant over time.  In the case of bounded functionals, the test supermartingales proposed in their Lemmas~3.1 and 3.2 reduce to the Capital test supermartingale of \Cref{sec:supg} up to some rescaling of the functions $m$ and $\ell$. Transposing their results to the setting of \Cref{sec:supg}, Theorem~4.2 and Proposition~4.3 of \cite{cassgrain22-anytime} guarantee that $\PP[\tau_\alpha<+\infty]=1$ if there exists $g \in \cG$ and $\lambda \in \Lambda$ such that 
$$
\liminf_{n\to+\infty} \frac{1}{n}\sum_{t=1}^n\log(1+ \lambda g(X_t)) > 0 \quad \PP\text{-a.s.} \;,
$$
which is possible only if $\sup_{g\in\cG}\frac{1}{n}\sum_{t=1}^n g(X_t)$ does not converge to $0$ as $n \to +\infty$. To this extent, our \Cref{thm:hoeffding-g-bound,thm:capital-g-bound} are stronger since they include larger alternatives. In addition, \cite{cassgrain22-anytime} only show asymptotic power while we also provide bounds on the expected rejection time.

\subsection{Comparison of forecasters}\label{sec:comparison}
In this section, we extend the work of \cite{Henzi21-comparing} to non binary forecasters and provide power guarantees under the alternative on the difference of stochastic regrets. Considering two predictable sequences $(\theta_t)_{t\in\nset}$ and $(\xi_t)_{t\in\nset}$ valued in $\Theta \subset \rset^d$ and an adapted sequence $(Y_t)_{t\in \nset}$ values in $\Yset$, we want to test the null hypothesis 
$$
\cH_0 : \forall t\geq 1, \,\PEarg[t-1]{\ell(\theta_t,Y_t) - \ell(\xi_t,Y_t)} \leq 0 \quad \PP\text{-a.s.} \;,
$$
for some known loss function $\ell : \Theta\times\Yset \to \cL\subset\rset$. When $\Theta = [0,1]$ and $\Xset = \lrc{0,1}$ this corresponds to the setting of \cite{Henzi21-comparing} with the null hypothesis defined in their Equation~(4) if we take $c_t = 1$ and $h = 1$. When $\cL$ is bounded, this reduces to the composite null of \Cref{sec:ineq_hyp} with $X_t = \ell(\theta_t,Y_t) - \ell(\xi_t,Y_t)$. In this case, it should be noted that the alternatives of \Cref{thm:power-hoeff-positif,thm:power-capital-positif} imply that the stochastic regret of $(\theta_t)_{t\in\nset}$ exceeds the one of $(\xi_t)_{t\in\nset}$ by at least $nm_n$ and we have seen that \Cref{thm:power-hoeff-positif,thm:power-capital-positif} respectively allow $nm_n$ to be of the order $\bigo{\sqrt{n\log(n)}}$ and $\bigo{\log(n)}$. Hence we can discriminate two forecasters even if both achieve logarithmic stochastic regret using Capital test supermartingale.

\section{Numerical simulations}\label{sec:simulation}
\subsection{Bounded mean testing}
In this section, we compare the power of the different test procedures introduced throughout the paper on simulated examples. To do so, we generate $T$ samples of a $d$-dimensional process $X:=(X_t)_{t =1,\cdots,T}$ with non-zero mean and compute the $T$ first steps of the test supermartingale $(W_t)_{t=1,\cdots, T}$ and the truncated rejection time $\tau_\alpha \wedge T$ at level $\alpha$. Replicating this procedure multiple times provides a Monte-Carlo estimate of $\PE[\tau_\alpha \wedge T]$ that can be used to compare the testing procedures. Throughout this section, we take $\alpha=0.05$ and $T=1000$ and the expected truncated rejection times are estimated using $500$ Monte-Carlo replicates.

\subsubsection{Experiment 1: One axis mean}
In the first experiment, we consider the $d$-dimensional process
$X_t =  (m t^{-a},0,\cdots,0)^\top + t^{-b} \epsilon_t$, 
for different values of $m \in (0,1/2)$, $a\in [0,1)$, $b\in [0,1)$ and $d\geq 2$ and where $(\epsilon_t)_{t\geq 1}$ is i.i.d drawn uniformly over the $\ell^2$-ball of $\rset^d$ with radius $1/5$.
Then \Cref{ass:bounded-y} holds with $B = 0.7$ and $D = 0.9$ and we have
$\norm{\mu_n}_\infty = \norm{\mu_n}_2 = m_n := \frac{m}{n}\sum_{t=1}^n t^{-a}$ and $\nu_{n,\infty} \leq \nu_{n,2} 
\leq v_n := \frac{1}{n}\sum_{t=1}^n \lr{m^2 t^{-2a} + \frac{t^{-2b}}{25}}$. 
In the stationary case where $a=b=0$, our theoretical bounds give
\begin{align*}
\PEarg{\tau_\alpha^{\rm H, FTL}} &\leq \bigo{\linlog\lr{\frac{1}{m^2}} + \frac{\log(1/\alpha)}{m^2}} \\  
\PE[\tau_{\alpha}^{\rm C,EWA}] &\leq \bigo{\linlog\lr{\frac{1}{(\epsilon m-4\epsilon^2 v)_+}} + \frac{\log(d/\alpha)}{(\epsilon m-4\epsilon^2 v)_+}} \\
\PE[\tau_{\alpha}^{\rm C, ONS}] &\leq \bigo{\linlog\lr{\frac{d}{m}}+\frac{d+\log(d/\alpha)}{m}} \\
\PE[\tau_{\alpha}^{\rm C, 2steps}] &\leq \bigo{\linlog\lr{\frac{1}{m}}+\frac{\log(1/\alpha)}{m}} 
\end{align*}
For CapitalEWA we take $\epsilon = 1/(2B)$ (the maximal possible value) even if the bound is infinite. In practice, we observe a finite bound.
For all procedures, the dependence with $m$ and $d$ seem consistent with the experimental rejection times shown in \Cref{fig:constant_mean_ball}. In particular, we observe that the Hoeffding and Capital2steps procedures are indeed independent of $d$.
In the non-stationary case where $a,b >0$, we have finite theoretical bounds for the Hoeffding and Capital2steps procedures  if $m_{n} \geq \bigo{\sqrt{\log(n)/n}}$ i.e. $a < 1/2$. For the CapitalEWA procedure, we need $v_{n} \leq \bigo{m_{n}}$ i.e. $b \geq a/2$ and for the CapitalONS, we need  $ v_{n} \leq \bigo{\frac{nm_{n}^2}{\log(n)}}$ i.e. $b\geq (a-1/2)_+$ or $a\wedge b \geq 1/2$. Furthermore, given the expression of $v_n$, we can expect lower dependence on $b$ when $b \geq a$ or $a\wedge b \geq 1/2$. \Cref{fig:decreasing_mean_ball} gathers the experimental rejection times as functions of $a$ and $b$. We observe, indeed, that the Hoeffding procedure can reject only when $a < 1/2$ and that for $b \geq a$ or $a\wedge b \geq 1/2$ all procedures have a limited dependence on $b$. However, we also see the limitation of our theoretical bounds as the other procedures have finite rejection times even in case which are not supported by our theoretical bounds. Finally, it is interesting to note that the CapitalONS procedure exhibits a stronger dependence on $b$ than the others.

\begin{figure}[h]
    \centering
    \begin{subfigure}{\textwidth}
    \centering
    \includegraphics[width=\textwidth]{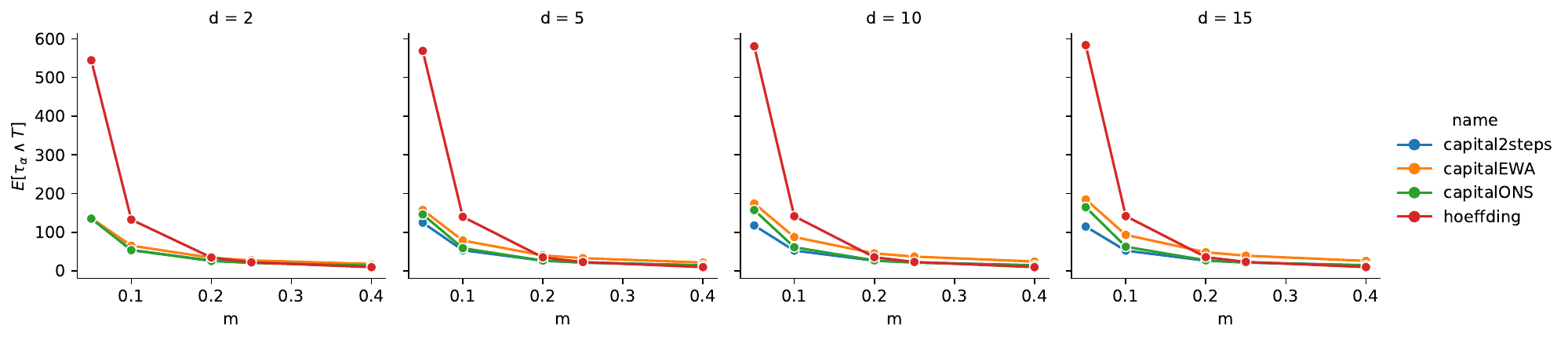}
    \caption{Evolution with $m$ for different values of $d$.}
    \label{fig:constant_mean_ball_m}
    \end{subfigure}
    \centering
    \begin{subfigure}{.75\textwidth}
    \centering
    \includegraphics[width=\linewidth]{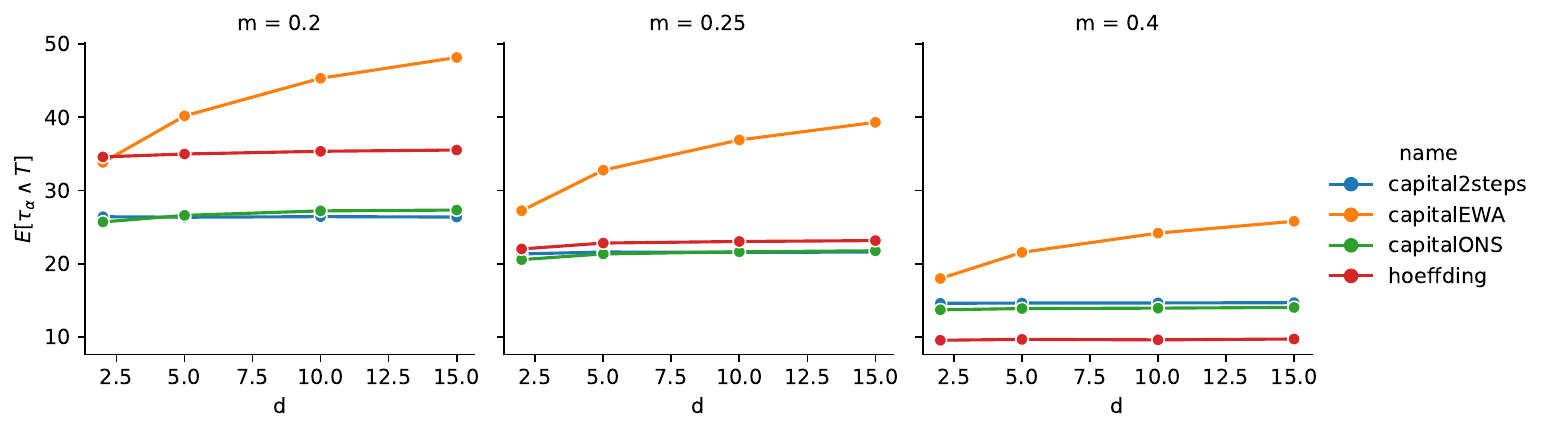}
    \caption{Evolution with $d$ for different values of $m$.}
    \label{fig:constant_mean_ball_d}
    \end{subfigure}
\caption{Truncated rejection times for Experiment 1 with $a=b=0$ (constant mean and variance).}
\label{fig:constant_mean_ball}
\end{figure}

\begin{figure}[H]
    \centering
    \begin{subfigure}{0.45\textwidth}
    \centering
    \includegraphics[width=\textwidth]{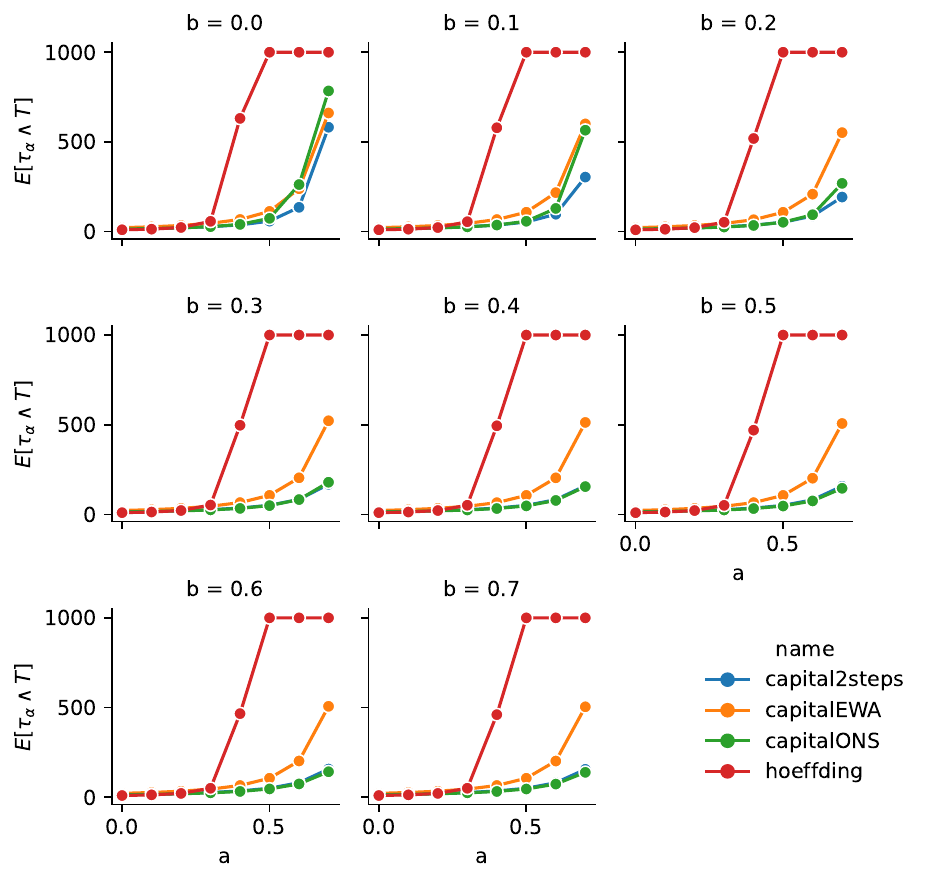}
    \caption{Evolution with $a$ for different values of $b$.}
    \label{fig:decreasing_mean_ball_a}
    \end{subfigure}
    \begin{subfigure}{0.45\textwidth}
    \centering
    \includegraphics[width=\linewidth]{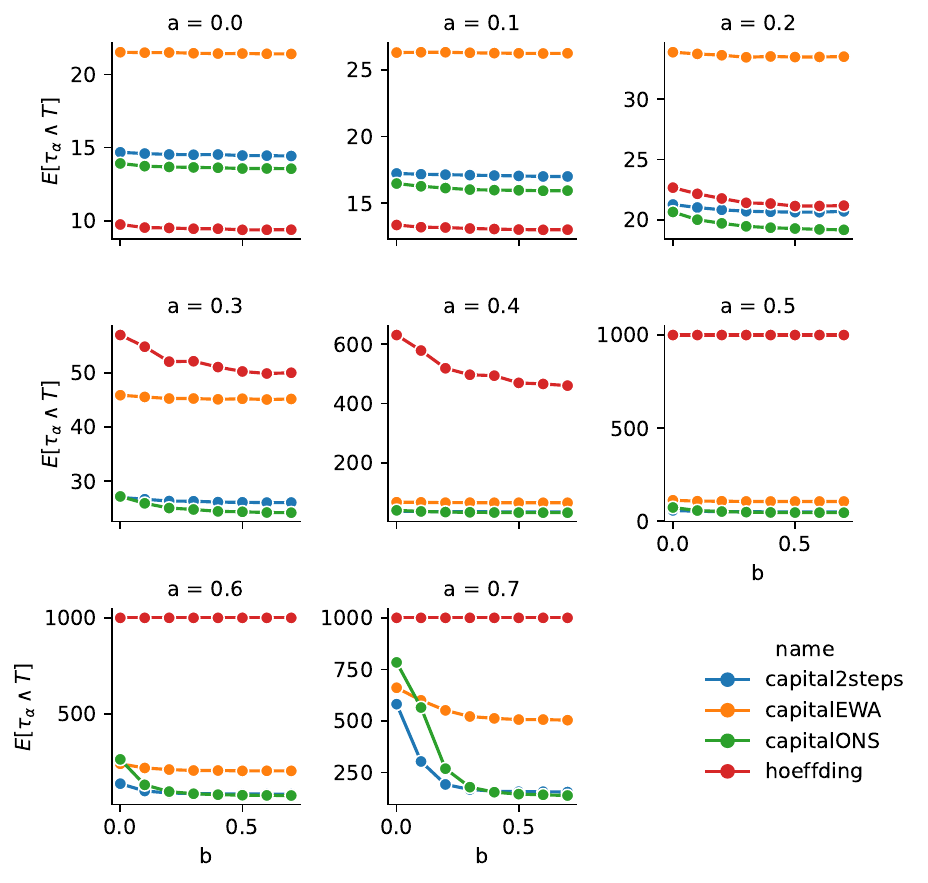}
    \caption{Evolution with $b$ for different values of $a$.}
    \label{fig:decreasing_mean_ball_b}
    \end{subfigure}
\caption{Truncated rejection times for Experiment 1 with $a\geq 0$, $b\geq 0$ (decreasing mean and variance) and $m=0.4$, $d=5$.}
\label{fig:decreasing_mean_ball}
\end{figure}

\subsubsection{Experiment 2: Spiral mean}
In the second experiment, we consider the process
$
X_t =  mt^{-a}(\cos(2\pi t/M), \sin(2\pi t/M))^\top  + t^{-2a} \epsilon_t 
$,
for $m=0.4$, with $a\in [0,1)$ and $M \in \nset^*$ and where $(\epsilon_t)_{t\geq 1}$ is i.i.d drawn uniformly over the $\ell^2$-ball of $\rset^2$ with radius $1/10$, see \Cref{fig:examples_X_simu2} for examples. 
In this case, \Cref{ass:bounded-y} holds with $B = 0.6$ and $D = 1.2$. 
As illustrated in \Cref{fig:spirale}, the lower $M$ and the higher $a$ are, the harder it is to reject the null because $\mu_n$ vanishes faster (see \Cref{fig:examples_X_simu2}).  We also observe that the CapitalONS procedure is more robust to complex cases with lower $M$ or higher $a$. In this setting, the CapitalONS procedure is a better strategy. However, this procedure necessitates to perform a projection at each step which is very time consuming compared to the other procedures. On average in this experiment, one iteration of the CapitalONS procedure takes 330ms compared to less that 0.1ms for the others on a MacBook Pro M1 with 8Go of RAM. Finally, it is interesting to note that, the rejection time does not seem to grow smoothly with $a$. On the contrary their always seems to be a breaking point below which the betting procedures reject the null very fast and above which the betting procedures fail to reject the null. 
To conclude our analysis, we propose to visualize the bets obtained by the different strategies in \Cref{fig:examples_X_simu2}. Interestingly, we observe very different behaviors. We observe that, for low values of $a$ and large values of $M$, the bets have more diverse directions for all strategies and the bets stay larger longer, hence giving more chance to reject the null. This behavior is more present for the capitalONS bets, even for larger values of $a$, thus explaining its better performances. 
\begin{figure}[p]
\centering
\begin{subfigure}{0.75\textwidth}
    \centering
    \includegraphics[width=\linewidth]{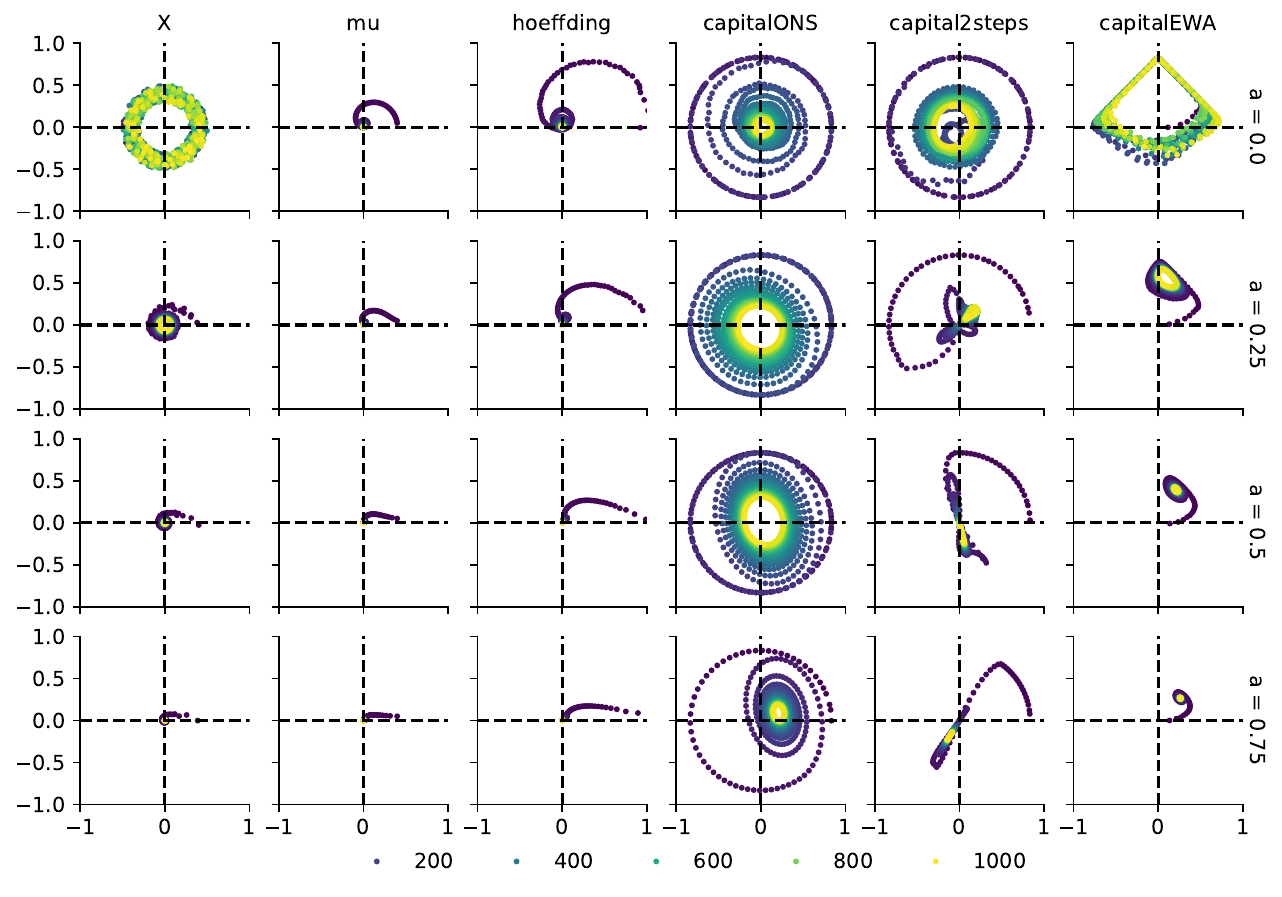}
    \caption{$M=50$}
\end{subfigure}
\begin{subfigure}{0.75\textwidth}
    \centering
    \includegraphics[width=\linewidth]{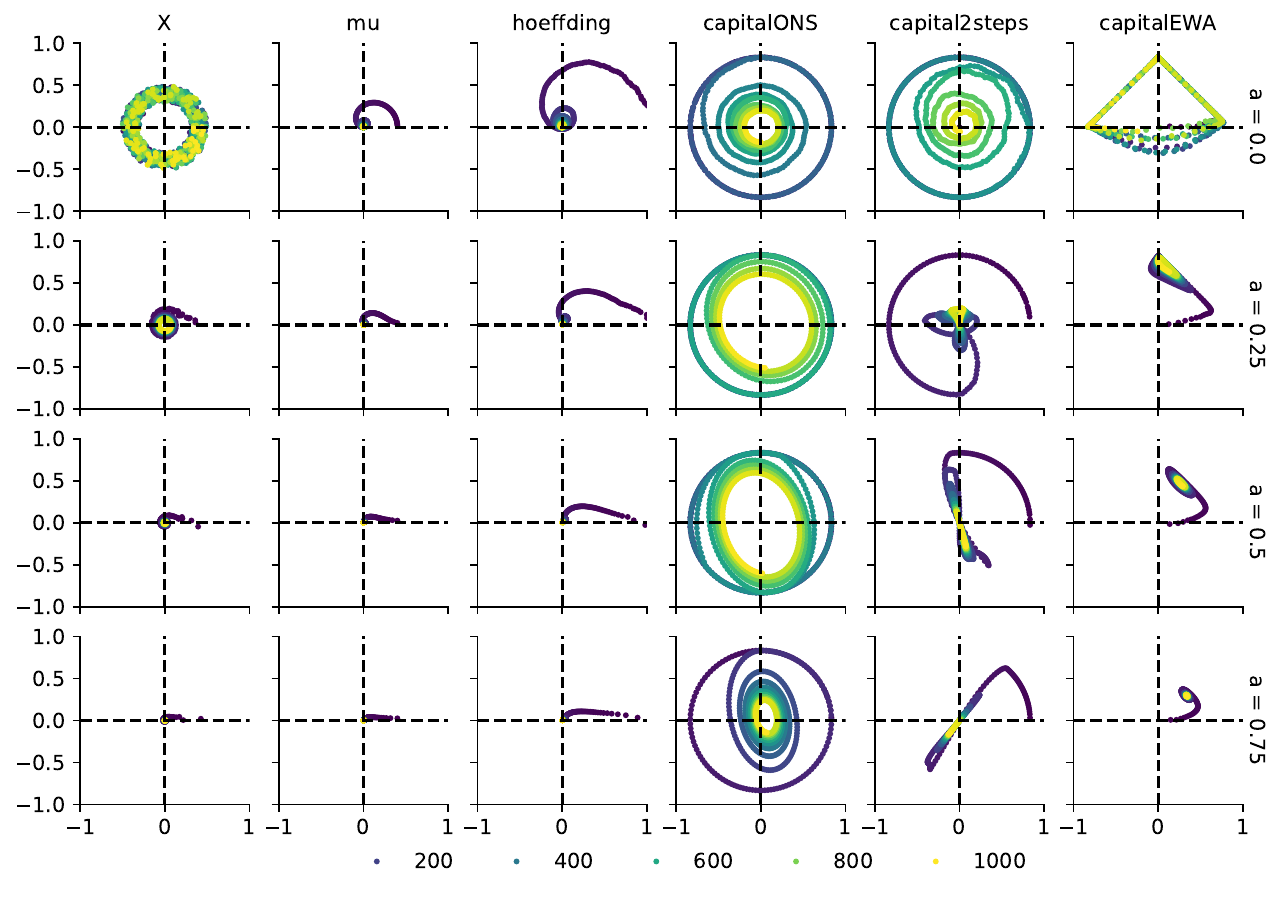}
    \caption{$M=100$}
\end{subfigure}
\caption{Examples of $X_t$ (first column),  $\mu_t$ (second column) and the bets obtained by the different strategies (other columns) in Experiment 2 for different values of $a$ (rows) and M.}
    \label{fig:examples_X_simu2}
\end{figure}
\begin{figure}[h]
    \centering
    \includegraphics[width=\linewidth]{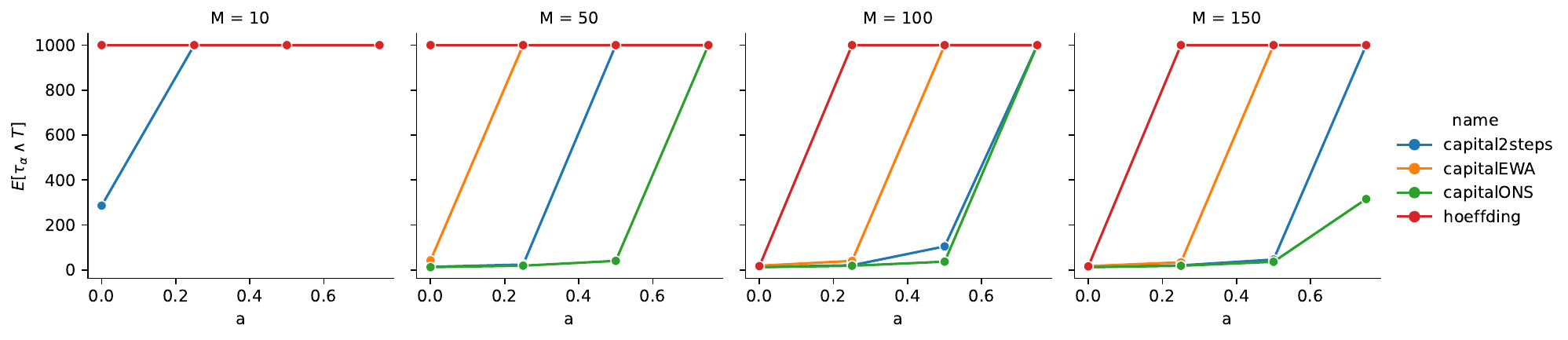}
    \caption{Evolution with $a$ of the truncated rejection time for Experiment 2 for different values of $M$.}
    \label{fig:spirale}
\end{figure}

\begin{figure}[h]
    \centering
\includegraphics[width=0.8\linewidth]{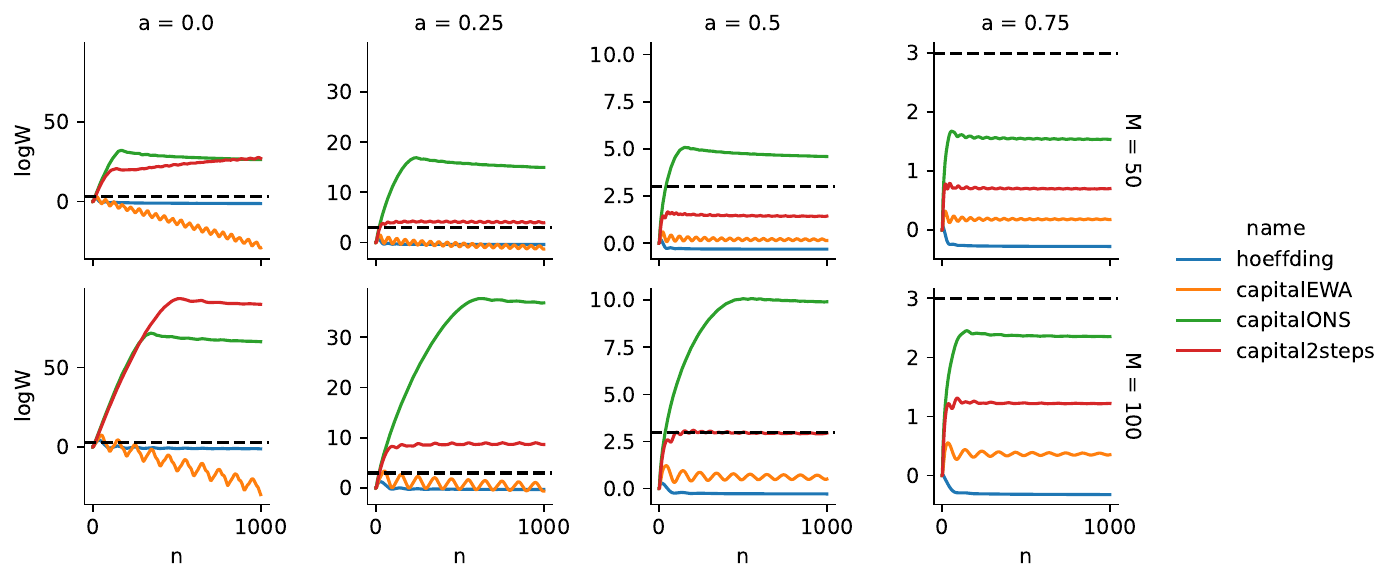}
    \caption{Examples of $\log(W_n)$ in Experiment 2 for different values of $a$ (columns) and $M$ (rows). Dashed horizontal line represents the rejection threshold $\log(1/\alpha)$. }
    \label{fig:spirale_logW}
\end{figure}

\subsection{Comparison of binary forecasters}
In this section, we reproduce the experiment of Section~4.2 in \cite{Henzi21-comparing} to compare their testing procedure with ours. We generate $Z_t = \epsilon_t + \theta \sum_{j=1}^4 \epsilon_{t-j}$ and take $Y_t = \1_{\lrc{Z_t > 0}}$. The two forecasters in competition are $p_t = \CPP{Z_t>0}{Z_{t-j},j=2,..,4}$ and $q_t = \CPP{Z_t>0}{Z_{t-j},j=1,..,4}$ so that $q_t$ outperforms $p_t$ so we expect to reject the null hypothesis
$$
\cH_0 : \forall t\geq 1, \,\PEarg[t-1]{\ell(p_t,Y_t) - \ell(q_t,Y_t)} \leq 0 \quad \PP\text{-a.s.} \;,
$$
where $\ell(p,y) = (p-y)^2$ is the Brier score. 
As seen in \Cref{sec:comparison}, this hypothesis can be tested online with the Hoeffding and Capital test supermartingales applied to $X_t = \ell(p_t,Y_t) - \ell(q_t,Y_t)$ with nonnegative bets. We propose to use the Hoeffding test supermatingale with FTL, the Capital test supermartingale with ONS and EWA, where the latter reduces to taking $\lambda = 1/2$ in the definition of $L_n^{\rm C}(\lambda)$.   
In \cite{Henzi21-comparing}, the authors introduce another supermartingale test whose betting strategy is optimized at each time step using the GRO criterion, which requires providing the distribution of $Y_t$ given $\cF_{t-1}$ under the alternative.  
In this experiment, we know the true distribution since $\pi_t:=\CPP{Y_t = 1}{\cF_{t-1}}=q_t$. However, in practice, choosing an appropriate distribution to compute the betting strategy can be challenging and the authors suggest taking a convex combination $\hat\pi_t = \beta p_t + (1-\beta) q_t$ with $\beta \in (0,1)$ where $\beta$ can be chosen using an a priori assumption on the alternative. To limit the dependence on this a priori knowledge, the authors also suggest a mixture strategy which consists in taking the mean of the supermartingales obtained for different $\beta$'s. On the contrary, the betting procedures studied in this paper do not rely on a priori on the alternative since the betting strategies optimize the GRO criterion with the empirical distribution for $\hat\pi_t$. This is a non-negligible advantage in practice. 
In \Cref{fig:comparison}, we compare the Hoeffding and Capital betting procedures with the one of \cite{Henzi21-comparing} for different values of $\beta$ and for mixture strategy obtained by taking the mean of the supermartingales obtained for these $\beta$'s. We compute the mean truncated rejection times for $500$ Monte-Carlo replicates of the experiment with maximum sample size $T=1000$. For the procedure of \cite{Henzi21-comparing}, we observe that  the best rejection time is obtained for $\beta=0$ (Henzi\_0) which is the true distribution and that the test looses power as $\beta$ grows and reaches zero power when $\beta \geq 0.5$. This shows that the choice of $\beta$ can change significantly the power of the testing procedure. The CapitalEWA and CapitalONS perform similarly to the mixture strategy of \cite{Henzi21-comparing}.  

\begin{figure}[h]
    \centering
    \includegraphics[width=.8\linewidth]{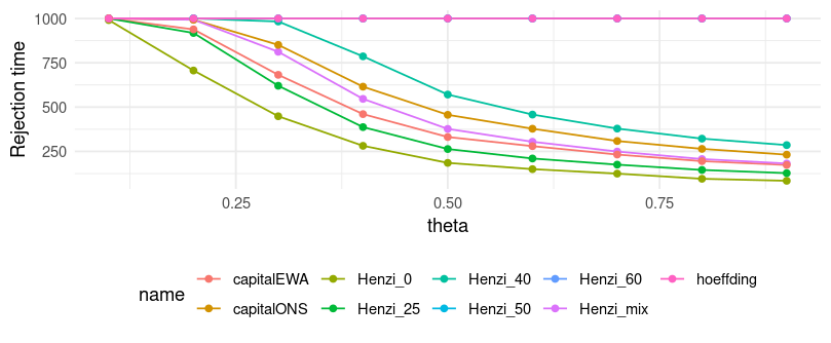}
    \caption{Truncated rejection time for comparison of forecasters. Henzi\_$x$ is the procedure of \cite{Henzi21-comparing} with $\hat\pi_t = (x/100) p_t + (1-x/100)q_t$ and Henzi\_mix is the mean of all the others.}
    \label{fig:comparison}
\end{figure}

\section{Conclusion}
In this paper, we conduct a theoretical and numerical comparison of various test martingales. We establish power properties under non-i.i.d. alternatives, extending beyond the existing literature. Notably, the Capital test supermartingale seems achieves a detection boundary of order $\bigo{\log(n)/n}$ with nearly-optimal rate.  This acceleration is attainable under specific conditions on the second-order properties of the alternative, particularly for betting strategies with low regret. 
Upper bounds on averaged stopping times and extensive numerical experiments do not yield conclusive comparisons between the EWA, ONS, and 2-steps betting strategies. In summary, ONS demonstrates the highest robustness to alternatives in high-dimensional settings, albeit at the cost of significant computational overhead. EWA, while much faster, suffers from a degradation in power properties when applied to complex multivariate alternatives. The 2-steps strategy appears to offer a balanced compromise, supporting the conclusions of \cite{Shekhar21}. 
Even in the most favorable deterministic scenarios, we demonstrate that acceleration is inherently limited due to the boundedness of the betting strategies. Consequently, we establish that our bounds are optimal in a certain sense. Key open questions remain, including the proof of power properties for the 2-steps strategy under fast-rate alternatives that we observe empirically.

\bibliographystyle{apalike}
\bibliography{biblio.bib}

\appendix

\section{Proofs of \Cref{sec:mean-test}}
Throughout this section, we define $\hat{\mu}_n := \frac{1}{n}\sum_{t=1}^n X_t$ for all $n \geq 1$.
\subsection{Preliminary results}
In this section, we provide preliminary lemmas which will be useful for the proofs of the main results.

\begin{lemma}\label{lem:pinelis}
Let $(X_t)_{t\in\nset}$ be an adapted  sequence of random variables valued in a subset of $\rset^d$  with diameter $D$. Then, for all $r > 0$ and $n\in\nset$, we have
$\PP[\norm{\hat\mu_n-\mu_n}_2 > r/n] \leq 2\exp\lr{-\frac{r^2}{2nD^2}}$.
Moreover, if $d = 1$, the same result holds with $D/2$ instead of $D$ and, if we remove the norm in the left-hand-side term, we can divide the right-and-side term by $2$. 
\end{lemma}
\begin{proof}
Apply Theorem~3.5 of \cite{pinelis94} to the $(2,1)$-smooth Banach space $\rset^d$ with $d_j = (X_j - \CPE{X_j}{\cF_{j-1}}) \1_{\lrc{j \leq n}}$. For $d=1$, this is the Azuma-Hoeffding inequality stated in Lemma~A.7 of \cite{Cesa-Bianchi-lugosi-games}.  
\end{proof}

\begin{lemma}\label{lem:argument_capital}
Let $(X_t)_{t\geq 1}$ be an $\rset^d$-valued stochastic process and $\Gamma_t = \set{\gamma \in \rset^d}{\abs{\gamma^\top X_t} \leq 1/2}$. Then for all $n,K \geq 1$ and $\gamma_1,\cdots,\gamma_K \in \bigcap_{t=1}^n \Gamma_t$, we have, with probability at least $1-1/n^2$, for all $1\leq k \leq K$,
\begin{equation}\label{eq:argument_capital}
\sum_{t=1}^n \lr{\gamma_k^\top X_t - (\gamma_k^\top X_t)^2} \geq \sum_{t=1}^n \lr{\PEarg[t-1]{\gamma_k^\top X_t} - 4 \PEarg[t-1]{(\gamma_k^\top X_t)^2}} - 2\log(Kn^2)\;.
\end{equation}
\end{lemma}
\begin{proof}
Let for all $n \geq 1$ and $k=1,\cdots,K$, 
\begin{equation}\label{eq:Bnk}
    B_{n,k} := \lrc{\sum_{t=1}^n \lr{\gamma_k^\top X_t - (\gamma_k^\top X_t)^2} \geq \sum_{t=1}^n \lr{\PEarg[t-1]{\gamma_k^\top X_t} - 4 \PEarg[t-1]{(\gamma_k^\top X_t)^2}} - 2\log(Kn^2)} \;,
\end{equation} 
so that we want to show $\PP[\bigcup_{k=1}^K B_{n,k}^c] \leq 1/n^2$.
Then, letting 
$$
Z_{k,t} := \exp\lr{\frac{\PEarg[t-1]{\gamma_k^\top X_t} - \gamma_k^\top X_t + (\gamma_k^\top X_t)^2 - 4\PEarg[t-1]{(\gamma_k^\top X_t)^2)}}{2} } \;,
$$
we have, for all $n,K \geq 1$,
\begin{align*}
\PP[\bigcup_{k=1}^K B_{n,k}^c]
= \PP[\max_{1\leq k \leq K} \prod_{t=1}^n Z_{k,t} > Kn^2] 
\leq \PP[\sum_{k=1}^K \prod_{t=1}^n Z_{k,t} > Kn^2] 
&\leq \frac{1}{Kn^2}\PE[\sum_{k=1}^K \prod_{t=1}^n Z_{k,t}] \\
&\leq \frac{1}{Kn^2} \sum_{k=1}^K \PE[\prod_{t=1}^n Z_{k,t}] \;,
\end{align*}
and the result follows if $\PEarg[t-1]{Z_{k,t}} \leq 1$ for all $t \geq 1$ and $k=1,\cdots,K$ because, in this case, $\PE[\prod_{t=1}^n Z_{k,t}] \leq \PEarg{\prod_{t=1}^{n-1} Z_{k,t} \PEarg[n-1]{Z_{k,t}}} \leq \PE[\prod_{t=1}^{n-1} Z_{k,t}]$ and recursively using this argument leads to $\PE[\prod_{t=1}^n Z_{k,t}] \leq 1$.

To conclude the proof, we now let $t \geq 1$ and $k \in \lrc{1,\cdots,K}$ and show that $\PEarg[t-1]{Z_{k,t}} \leq 1$. From Lemma~B.1 of \cite{bercu-touati-08} and using the arguments of the proof of Proposition~3.1 of \cite{wintenberger2023soco}, we have that
$$
\PEarg[t-1]{\exp\lr{s(Y_t-\PEarg[t-1]{Y_t}) - s^2 (\PEarg[t-1]{Y_t^2} + Y_t^2)}} \leq 1 \;,
$$
holds for any $s \in \rset$ and any random variable $Y_t \in \rset^\nset$. 
Applying this result to $Y_t= \gamma_k^\top X_t$ and $s = - 1$ gives 
$$
\PEarg[t-1]{\exp\lr{\PEarg[t-1]{\gamma_k^\top X_t} - \gamma_k^\top X_t - (\gamma_k^\top X_t)^2- \PEarg[t-1]{(\gamma_k^\top X_t)^2}}} \leq 1 \;.
$$
On the other hand, Applying Lemma~A.3 of \cite{Cesa-Bianchi-lugosi-games} with $s=1/2$ and $X = 4(\gamma_k^\top X_t)^2 \in [0,1]$ yields
$$
\PEarg[t-1]{\exp\lr{2 (\gamma_k^\top X_t)^2-3\PEarg[t-1]{(\gamma_k^\top X_t)^2}}} \leq 1 \;,
$$
where we have used that $4(e^{1/2}-1) \leq 3$. Hence, the Cauchy-Schwarz inequality and the inequalities of the two previous displays give
\begin{align*}
\PEarg[t-1]{Z_{k,t}} 
&\leq \sqrt{\PEarg[t-1]{\rme^{\PEarg[t-1]{\gamma_k^\top X_t} - \gamma_k^\top X_t - (\gamma_k^\top X_t)^2- \PEarg[t-1]{(\gamma_k^\top X_t)^2}}}} \sqrt{\PEarg[t-1]{\rme^{2 (\gamma_k^\top X_t)^2-3\PEarg[t-1]{(\gamma_k^\top X_t)^2}}} } \\ 
&\leq 1 \;.    
\end{align*}
This concludes the proof.
\end{proof}

\subsection{Proofs of \Cref{sec:lower-bounds}}
\begin{proof}[Proof of \Cref{prop:limit-case-hoeffding}]
Take $d = 1$ and consider a deterministic process $X_t = z_t - z_{t-1} \geq 0$ for all $t \geq 1$ where  $(z_t)_{t \in \nset}$ is a  deterministic non-decreasing sequence with $z_0 = 0$. In this case $\norm{\mu_n}_2 = \frac{z_n}{n}$ and for all $n\geq 1$ we have 
$$
\log W_n^{\rm H} = \max_{\lambda \in \rset^d} \log L_n^{\rm H}(\lambda) - \cR_n
= \frac{2 z_n^2}{n D^2} - \cR_n 
< \log(1/\alpha) - \cR_n \;,
$$
where the last inequality holds if we take $z_n = m \sqrt{n}$ with $m < D / \sqrt{2\log(1/\alpha)}$. Assertion~\ref{itm:limit-case-hoeffding-1} follows by \Cref{ass:regret-hoeffding}. Similarly, we get Assertion~\ref{itm:limit-case-hoeffding-2} by taking $z_n = n m_n = \smallo{\sqrt{n}}$.
\end{proof}

\begin{proof}[Proof of \Cref{prop:hoeff-lower-power}]
    We take the same process $(X_t)_{t\in \nset}$ as in the proof of \Cref{prop:limit-case-hoeffding} with $z_n = nm$ so that $\log W_n^{\rm H} \leq \frac{2 n m^2}{D^2} - \cR_n$ for all $n \geq 1$ and $\log W_n^{\rm H} \geq \log(1/\alpha)$ is possible only if $n \geq \frac{D^2\log(1/\alpha)}{2m^2}$. 
\end{proof}

\begin{proof}[Proof of \Cref{prop:limit-case-capital}]
Take $d = 1$ and consider a deterministic process $X_t = z_t - z_{t-1} \geq 0$ for all $t \geq 1$ where $(z_t)_{t \in \nset}$ is a  deterministic non-decreasing sequence with $z_0 = 0$. In this case $\norm{\mu_n}_\infty = \frac{z_n}{n}$ and for all $n\geq 1$ we have 
$$
\log W_n^{\rm C} = \sum_{t=1}^n \log(1+\gamma_t X_t) \leq \sum_{t=1}^n \gamma_t X_t \leq \frac{1}{2B} \sum_{t=1}^n X_t = \frac{z_n}{2B} \;.
$$
Hence, taking $z_n = m n$ and $m < 2 B \log(1/\alpha)$, we get that $\log W_n^{\rm C} < \log(1/\alpha)$ for all $n > 1$ and Assertion~\ref{itm:limit-case-capital-1} follows. Similarly, we get Assertion~\ref{itm:limit-case-capital-2} by taking $z_n = n m_n = \smallo{1}$.
\end{proof}

\begin{proof}[Proof of \Cref{prop:capital-lower-power}]
    We take the same process $(X_t)_{t\in \nset}$ as in the proof of \Cref{prop:limit-case-capital} with $z_n = nm$ so that $\log W_n^{\rm H} \leq \frac{nm}{2B}$ for all $n \geq 1$. Hence $\log W_n^{\rm H} \geq \log(1/\alpha)$ is possible only if $n \geq \frac{2B\log(1/\alpha)}{m}$. 
\end{proof}

\subsection{Proofs of \Cref{sec:power-general}}

\begin{proof}[\bf Proof of \Cref{thm:power-hoeff-nodim}]
Define 
$A_n := \lrc{\norm{\mu_n - \hat\mu_n}_2 \leq 2D\sqrt{\log(n)/n}}$ and $B_n := \lrc{\norm{\mu_n}_2 \geq m_n}$ for all $n\geq 1$.
By \Cref{lem:pinelis}, we have $\PP[A_n^c] \leq 2/n^2$ and by definition we have $\varrho = \sum_{n\geq 1} \PP[B_n^c]$. Then letting $G_n := \lrc{\cR_n\leq r_n}$, we get from the definition of $\cR_n$ and the inequality $\norm{\hat\mu_n}_2 \geq \lr{\norm{\mu_n}_2 - \norm{\mu_n-\hat\mu_n}_2}_+$, that for all $n \geq 1$, on $G_n \cap A_n \cap B_n$,
$$
\log W_n^{\rm H} \geq \frac{2n\norm{\hat\mu_n}_2^2}{D^2} - r_n  
\geq \frac{2n\lr{m_n - 2D\sqrt{\log(n)/n}}_+^2}{D^2}- r_n  = u_n \;.
$$
Hence, letting $E_n := \lrc{\log W_n^{\rm H} \geq u_n}$, we have $\PP[E_n^c \cap G_n \cap A_n \cap B_n] = 0$ for all $n\geq 1$, and therefore
$$
\sum_{n\geq 1} \PP[E_n^c] 
\leq \rho + \varrho + \sum_{n \geq 1} \PP[E_n^c \cap G_n \cap B_n] 
\leq \rho + \varrho + \sum_{n \geq 1} \PP[A_n^c]  
\leq \rho + \varrho + \pi^2/3 \;.
$$
\end{proof}

\begin{proof}[\bf Proof of \Cref{thm:capital-power}] 
Using the fact that  $\log(1+x) \geq x-x^2$ for any $x \geq  -1/2$, we get that for all $n \geq 1$  and $\gamma \in \Gamma$, 
\begin{equation}\label{eq:bound-W-multi}
\log(W_n) \geq \max_{\gamma \in \Gamma} \lrc{\sum_{t=1}^n \gamma^\top X_t -  \sum_{t=1}^n (\gamma^\top X_t)^2}  - \cR_n \;.  
\end{equation}
Fix $n \geq 1$ and define for all $k=1,\cdots, 2d$, 
$C_n := \lrc{\nu_{n,\infty} \leq v_n}$, 
$D_{n,k}:= \lrc{\rme_k^\top \mu_n \geq m_n}$ and $G_n := \lrc{\cR_n \leq r_n}$.
Let also $D_n := \bigcup_{k=1}^{2d} D_n(\rme_k)$, so that $\lrc{\norm{\mu_n}_{\infty} \geq m_n} \subset D_n$ and $\sum_{n\geq 1} \PP[(C_n \cap D_n)^c] \leq \varrho$.  
Take now $\epsilon_n \in \Gamma$ so that for all $k=1,\cdots,2d$, $\gamma_{n,k} := \epsilon_n \rme_k \in \Gamma$ and define $B_{n,k}$ by \eqref{eq:Bnk}. Then, \Cref{lem:argument_capital} implies that $\PP[\bigcup_{k=1}^{2d} B_{n,k}^c] \leq 1/n^2$ and \eqref{eq:bound-W-multi} gives that, on $G_n \cap B_{n,k}\cap C_n \cap D_{n,k}$, we have
\begin{align*}
\log(W_n) 
&\geq \sum_{t=1}^n (\PEarg[t-1]{\gamma_{n,k}^\top X_t} - 4 \PEarg[t-1]{(\gamma_{n,k}^\top X_t)^2}) - r_n - 2\log(2dn^2) \\
&\geq \gamma_{n,k}^\top \lr{\sum_{t=1}^n \PEarg[t-1]{X_t}} - 4 \norm{\gamma_{n,k}}_1^2 \sum_{t=1}^n \PEarg[t-1]{\norm{X_t}_\infty^2} - r_n -2\log(2dn^2) \\
&\geq n\epsilon_n m_n - 4n\epsilon_n^2 v_n - r_n - 2\log(2dn^2) = u_n\;.
\end{align*}
Hence letting $E_n := \lrc{\log W_n \geq u_n}$, we have shown that $\PP[E_n^c \cap G_n \cap B_{n,k}\cap C_n \cap D_{n,k}] = 0$, for all $n \geq 1$ and $k\in \lrc{1,\cdots,2d}$. Finally, we get
\begin{align*}
\sum_{n\geq 1} \PP[E_n^c] 
\leq \rho + \varrho + \sum_{n\geq 1} \PP[E_n^c \cap G_n \cap C_n \cap D_n] 
&\leq \rho + \varrho + \sum_{n\geq 1} \PP[\bigcup_{k=1}^{2d} E_n^c \cap G_n \cap C_n \cap D_{n,k}] \\
&\leq \rho + \varrho + \sum_{n\geq 1}  \PP[\bigcup_{k=1}^{2d} B_{n,k}^c]\;.
\end{align*}
This concludes the proof of the first par since $\sum_{n\geq 1}  \PP[\bigcup_{k=1}^{2d} B_{n,k}^c] \leq \pi^2/6$. Then the values of $u_n$ defined in \eqref{eq:un-capital-aggregation} and \eqref{eq:un-capital} are respectively obtained  by taking $\epsilon_n=\epsilon$ and taking $\epsilon_n = \frac{1}{2B} \wedge \frac{m_n}{8v_n}$ and using the fact that $m_n - 4\epsilon_n v_n \geq m_n/2$.  

\end{proof}

\begin{proof}[\bf Proof of \Cref{thm:capital-2step-bound}]
    This result is obtained by taking $\cG = \set{x \mapsto u^\top x}{u \in \cB_{1/B}^d}$  in the setting of \Cref{sec:supg} and 
 \Cref{thm:capital-g-bound} which we prove in \Cref{sec:proofs_extensions}.
\end{proof}

\subsection{Proofs of \Cref{sec:explicit-bounds-d}}

We start by proving all the regret bounds.
\begin{proof}[Proof of \Cref{lem:regret-ftl}]
Let $f_t(\lambda) = \frac{\norm{\lambda}_2^2D^2}{8} - \lambda^\top Y_t$, then since $\lambda_{t+1} = \frac{4\hat\mu_t}{D^2} = \frac{4 Y_t}{D^2 t} + \lr{1-\frac{1}{t}}\lambda_t$, we get that
\begin{align*}
     f_t(\lambda_t) - f_t(\lambda_{t+1})
= \frac{D^2}{8}\lr{\norm{\lambda_t - \frac{4 Y_t}{D^2}}_2^2 - \norm{\lambda_{t+1} - \frac{4Y_t}{D^2}}_2^2} 
&= \frac{D^2}{8}\norm{\lambda_t - \frac{4Y_t}{D^2}}_2^2 \lr{1-\lr{1-\frac{1}{t}}^2} \\
&\leq \frac{4}{D^2 t}\norm{\hat\mu_{t-1} - Y_t}_2^2 \\ 
&\leq \frac{4}{t} \;,
\end{align*}
where the last inequality comes from the fact that $\hat\mu_{t-1}$ and $Y_t$ are in the same subset of diameter $D$ by \Cref{ass:bounded-y}. 
Hence, by Lemma~3.1 of \cite{Cesa-Bianchi-lugosi-games}) we have 
$$
\max_{\lambda\in\rset^d} \log L_n^{\rm H}(\lambda) - \log W_n^{\rm H, FTL} \leq \sum_{t=1}^n \lr{f_t(\lambda_t) - f_t(\lambda_{t+1})}
\leq 4 \sum_{t=1}^n \frac{1}{t} \leq 4 (1 + \log(n)) \;.
$$
\end{proof}
\begin{proof}[Proof of \Cref{lem:regret-ewa}]
Apply Proposition~3.1 of \cite{Cesa-Bianchi-lugosi-games} to $\ell(\gamma,x) =  -\log(1+\gamma^\top x)$ which is $1$-exp-concave in its first argument.  
\end{proof}
\begin{proof}[Proof of \Cref{lem:regret-ons}]
 Lemma~17 of \cite{cutkosky2018blackbox} gives that 
 $$
 \max_{\norm{\gamma}_2\leq 1/(2B)} \log L_n^{\rm C}(\gamma) - \log W_n^{\rm C} \leq d\lr{\frac{\beta}{8} + \frac{2}{\beta}\log(1+4n)} \;,
 $$ with $\beta = \frac{2-\log(3)}{2}$ and we conclude using the fact that $\log(1+4n) \leq \log(5n) = \log(5) + \log(n)$ and evaluating the constants. For the case where $d = 1$, $\Gamma = [0,1/(2B)]$ and $S = [0,1/2]$, the proof of  Lemma~17 of \cite{cutkosky2018blackbox} can be easily translated. 
\end{proof}

\begin{proof}[Proof of \Cref{lem:oga-stoch-regret}]
For all $\eta \in \cB_{1/2B}^d$, we have 
$$
\sum_{t=1}^n \PEarg[t-1]{\eta^\top X_t} - \sum_{t=1}^n \PEarg[t-1]{\eta_t^\top X_t}  = \sum_{t=1}^n (\eta-\eta_t)^\top X_t + \sum_{t=1}^n (\eta - \eta_t)^\top (\PEarg[t-1]{X_t} - X_t)\;,
$$
where the first sum is the regret of the OGA algorithm and is bounded by $\sqrt{n}$ (see \cite[Section~A.4]{Shekhar21}) and the second is a martingale with bounded differences and is therefore bounded by $2\sqrt{n\log(n)}$ with probability at least $1-1/n^2$ by  \cite[Lemma~A.7]{Cesa-Bianchi-lugosi-games}
\end{proof}

Now, note that all the betting strategies used in \Cref{cor:hoeffding-examples-multi,cor:capital-examples-aggregation-multi,cor:capital-examples-multi} achieve regrets bounded by $r\log(n) + r'$ where $r=r'=4$ for \Cref{cor:hoeffding-examples-multi}, $r=0$ and $r'=\log(2d)$ for \Cref{cor:capital-examples-aggregation-multi} and $r=4.5d$ and $r'=7.2d$ for \Cref{cor:capital-examples-multi}. 
With this common form of regret, the proofs of \Cref{cor:hoeffding-examples-multi,cor:capital-examples-aggregation-multi,cor:capital-examples-multi} now reduce to bounding $\aleph((u_n)_{n\geq 1}, x)$ with $r_n = r\log(n)$ and $x=\log(1/\alpha)+r'$ which we do for each corollary. 
The proofs rely on the following lemma
\begin{lemma}\label{lem:aleph-examples}
The following assertions hold.
\begin{enumerate}
    \item\label{itm:aleph-comparison} Assume that $u_n \geq u_n^{(1)} \1_{\lrc{n<n_0}} + u_n^{(2)} \1_{\lrc{n\geq n_0}}$. Then for all $x\in\rset$,
    $$
    \aleph\lr{(u_n)_{n\geq 1},x} \leq \lr{n_0\wedge \aleph\lr{(u_n^{(1)})_{n\geq 1},x}} \vee \aleph\lr{(u_n^{(2)})_{n\geq 1},x}\;.
    $$
\item\label{itm:aleph-nbeta-log} For all $z \in \rset$, define $\cL(z)$ is the unique solution of $\log(y)/y = z$ on $[\rme, +\infty)$ when $z \leq 1/e$ and equals $0$ otherwise. Then, for any $a,b, \beta > 0$ and $x\in \rset$ we have 
$$
\aleph((an^{\beta} - b\log(n))_{n\geq 1},x) = \ceil{\lr{\rme^{-\beta x /b} \cL\lr{\frac{a\beta}{b}\rme^{-\beta x/b}}}^{1/\beta}} \leq \ceil{2^{1/\beta}\lr{\linlog\lr{\frac{b}{a\beta}} + \frac{x}{a}}^{1/\beta}} \;. $$
\end{enumerate}
\end{lemma}
\begin{proof}
To prove Assertion~\ref{itm:aleph-comparison}, let $\aleph_i = \aleph\lr{(u_n^{(i)})_{n\geq 1},x}$ so that we need to prove that $u_n \geq x$  for all $n \geq n_1 := (n_0 \wedge \aleph_1) \vee \aleph_2$. First assume that $\aleph_1 \geq n_0$. Then $n_1 = n_0 \vee \aleph_2$ and for all $n \geq n_1$, we have $u_n \geq u_n^{(2)} \geq x$. Now, assume that $\aleph_1 \leq n_0$. Then $n_1 = \aleph_1 \vee \aleph_2$ and for all $n \geq n_1$, we have $u_n \geq u_n^{(1)} \wedge u_n^{(2)} \geq x$. This concludes the proof of Assertion~\ref{itm:aleph-comparison}. 
To prove Assertion~\ref{itm:aleph-nbeta-log}, note that $an^\beta - b \log(n) \geq x$ if and only if $\frac{\log(n^\beta \rme^{\beta x/b})}{n^\beta \rme^{\beta x/b}} \leq \frac{a\beta}{b}\rme^{-\beta x/b}$ and the first equality follows. The second inequality follows from the fact that, for any $z \leq \rme^{-1}$, we have $z = \frac{\log(\cL(z))}{\cL(z)} \leq \frac{1}{\sqrt{\cL(z)}}$ which implies that  $\cL(z)\leq \frac{1}{z^2}$ and finally $\cL(z) = \frac{\log(\cL(z))}{z} \leq \frac{2 \log(1/z)}{z}$.
\end{proof}

\begin{proof}[\bf Proof of \Cref{cor:hoeffding-examples-multi}]  For Assertion~\ref{itm:hoeffding-examples-multi-1}, we have $u_n =  \frac{2n}{D^2}\lr{mn^{-a} - 2D\sqrt{\log(n)/n}}_+^2- r \log(n)$ which is greater than $\frac{m^2}{2D^2}n^{1-2a} - r \log(n)$ if $n \geq \aleph\lr{\lr{m^2 n^{1-2a}-16 D^2\log(n)}_{n\geq 1},0}$. Hence, by \Cref{lem:aleph-examples}, we get
    $$
    \aleph((u_n)_{n\geq 1}, x) \leq \ceil{\lr{2\linlog\lr{\frac{16D^2}{m^2(1-2a)}}}^{1/(1-2a)}} \vee \ceil{\lr{2\linlog\lr{\frac{2rD^2}{m^2(1-2a)}} + \frac{4D^2x}{m^2}}^{1/(1-2a)}} \;.
    $$
    For Assertion~\ref{itm:hoeffding-examples-multi-2}, we have $u_n = \lr{2(m/D - 2)^2- r}\log(n)$ and the result follows easily. 
\end{proof}

\begin{proof}[\bf Proof of \Cref{cor:capital-examples-aggregation-multi}] 
We bound $\aleph := \aleph((u_n)_{n\geq 1}, x)$   for each case. 
\begin{itemize}
    \item If $v_n=vn^{-a}$, we get $u_n \geq \epsilon (m-4\epsilon v)n^{1-a} - 2\log(2d) - (r+4)\log(n)$ and \Cref{lem:aleph-examples} gives that, if $\epsilon < \frac{m}{4v}$,
    $$
\aleph \leq 
 \ceil{\lr{2\linlog\lr{\frac{r+4}{\epsilon(m-4\epsilon v)(1-a)}} + \frac{2(x+2\log(2d))}{\epsilon(m-4\epsilon v)}}^{1/(1-a)}} \;.
$$
\item If $v_n = vn^{-2b}$ with $a/2 < b < 1/2$, then for all $n \geq \ceil{\lr{\frac{8\epsilon v}{m}}^{1/(2b-a)}}$ we have $u_n \geq \frac{\epsilon m n^{1-a}}{2} -2\log(2d) - (r+4)\log(n)$ and \Cref{lem:aleph-examples} gives
$$
\aleph \leq 
 \ceil{\lr{\frac{8\epsilon v}{m}}^{1/(2b-a)}} \vee \ceil{\lr{2\linlog\lr{\frac{2(r+4)}{\epsilon m(1-a)}} + \frac{4(x+2\log(2d))}{\epsilon m}}^{1/(1-a)}} \;.
$$
\item If $v_n = vn^{-1}$, then $u_n \geq \epsilon mn^{1-a} - 4\epsilon^2 v - 2\log(2d) - (r+4)\log(n)$ and \Cref{lem:aleph-examples} gives 
$$
\aleph \leq 
\ceil{\lr{2\linlog\lr{\frac{r+4}{\epsilon m(1-a)}} + \frac{2\lr{x + 2\log(2d) + 4\epsilon^2 v}}{\epsilon m}}^{1/(1-a)}} \;.
$$
\item If $v_n = v\log(n)/n$, then $u_n \geq \epsilon m n^{1-a} - \lr{4\epsilon^2 v+r+4}\log(n) - 2\log(2d)$ and \Cref{lem:aleph-examples} gives
$$
\aleph \leq 
\ceil{\lr{2\linlog\lr{\frac{(r+4+4\epsilon^2 v)}{\epsilon m(1-a)}} + \frac{2(x+2\log(2d))}{\epsilon m}}^{1/(1-a)}} \;.
$$
\end{itemize}
\end{proof}

\begin{proof}[\bf Proof of \Cref{cor:capital-examples-multi}]
    We have $u_n \geq u_n^{(1)} \wedge u_n^{(2)}$ with $u_n^{(1)} = \frac{mn^{1-a}}{4B} - (r+4)\log(n)-2\log(2d)$ and $u_n^{(2)} = \frac{m^2n^{1-2a}}{16v_n}  - (r+4)\log(n)-2\log(2d)$. Let $\aleph_1 = \aleph\lr{(u_n^{(1)})_{n\geq 1}, x}$, $\aleph_2 = \aleph\lr{(u_n^{(2)})_{n\geq 1}, x}$ and $n_1 = \inf\set{n \geq 1}{\forall k \geq n, u_k^{(1)} \geq u_k^{(2)}}$, $n_2 = \inf\set{n \geq 1}{\forall k \geq n, u_k^{(2)} \geq u_k^{(1)}}$. Then \Cref{lem:aleph-examples} gives 
            $$
\aleph\lr{(u_n)_{n\geq 1}, x} = 
\begin{cases}
    \aleph_1 \vee (n_1 \wedge \aleph_{2}) & \text{if } n_1 < +\infty \\
    \aleph_2 \vee (n_2 \wedge \aleph_{1}) & \text{if } n_2 < +\infty 
\end{cases} \;.
$$
By \Cref{lem:aleph-examples}, we have 
$
\aleph_1 \leq \ceil{\lr{2\linlog\lr{\frac{4B(r+4)}{m(1-a)}} + \frac{8B(x+2\log(2d))}{m}}^{1/(1-a)}} \;.
$
We compute the other terms for the different values of $v_n$.
\begin{itemize}
\item If $v_n = vn^{-2b}$, we have 
$u_n^{(2)} = \frac{m^2}{16v}n^{1-2(a-b)}  - 2\log(2d) - (r+4)\log(n)$ and \Cref{lem:aleph-examples} gives that, if $b > a-1/2$, 
$
\aleph_2 \leq \ceil{\lr{2\linlog\lr{\frac{16v(r+4)}{m^2(1-2(a-b))}} + \frac{32v(x+2\log(2d))}{m^2}}^{\frac{1}{1-2(a-b)}}}\;. 
$
We also have 
$
n_1 = \ceil{\lr{\frac{Bm}{4v}}^{\frac{1}{(a-2b)_+}}}$ and $n_2 = \ceil{\lr{\frac{4v}{Bm}}^{\frac{1}{(2b-a)_+}}}$. 
\item If $v_n = v\log(n)/n$, we have, for all $n\geq 3$, since  $\log(n) \leq \log(n)^2$,
$$
u_n^{(2)} + 2\log(2d)
= \frac{m^2n^{2(1-a)}}{16v\log(n)}  - (r+4)\log(n) 
\geq \frac{1}{\log(n)^2}\lr{{\frac{m^2n^{2(1-a)}}{16v}  - (r+4)\log(n)^2}} \;.
$$
Hence
\begin{align*}
\aleph_2 
&\leq 3 \vee \aleph\lr{\lr{\frac{m^2n^{2(1-a)}}{16v}-(r+4+x+2\log(2d))\log(n)^2}_{n\geq 1}, 0}    \\
&= 3 \vee \aleph\lr{\lr{\frac{mn^{1-a}}{4\sqrt{v}}-\sqrt{r+4+x+2\log(2d)}\log(n)}_{n\geq 1}, 0}    \\
&\leq 3 \vee \ceil{\lr{2\linlog\lr{\frac{2\sqrt{v(r+4+x+2\log(2d))}}{m(1-a)}}}^{1/(1-a)}} \;,
\end{align*}
by \Cref{lem:aleph-examples}. We also 
have $n_2 = \ceil{\lr{2\linlog\lr{\frac{4v}{Bm(1-a)}}}^{1/(1-a)}}$.
\end{itemize}
\end{proof}

\begin{proof}[\bf Proof of \Cref{cor:capital-examples-2step}]   
\Cref{lem:oga-stoch-regret} gives that we can take $s_n = 5\sqrt{n\log(n)}$ for $n \geq 3$. Then, for all $n \geq \aleph_0 := 3\vee \aleph\lr{\lr{nm_n-10\sqrt{n\log(n)}}_{n\geq 1},0} = \ceil{\lr{2\linlog\lr{\frac{100}{m^2(1-2a)}}}^{1/(1-2a)}}$, we have $u_n \geq u_n'$  with $u_n' := \frac{mn^{1-a}}{8} \lr{1 \vee \frac{m n^{1-a}}{8nv_n}} - r\log(n)-r'$. Hence,
$\aleph\lr{(u_n)_{n\geq 1}, \log(1/\alpha)} \leq \aleph_0 \vee \aleph\lr{(u_n')_{n\geq 1}, \log(1/\alpha)}$,
where the second term is computed as in the proof of \Cref{cor:capital-examples-multi} with different constants
\end{proof}

\section{Proofs of \Cref{sec:extensions}}\label{sec:proofs_extensions}
\begin{proof}[Proof of \Cref{thm:power-hoeff-positif}]
The proof follows the same steps as the proof of \Cref{thm:power-hoeff-nodim} with $A_n := \lrc{\mu_n - \hat\mu_n \leq D\sqrt{\log(n)/n}}$  and $B_n := \lrc{\mu_n \geq m_n}$ and using \cite[Lemma~A.7]{Cesa-Bianchi-lugosi-games} instead of \Cref{lem:pinelis} and the fact that $\max_{\lambda \geq 0} \log L_n^{\rm H}(\lambda) = \frac{2n(\hat{\mu}_n)_+^2}{D^2}$. 
\end{proof}

\begin{proof}[Proof of \Cref{thm:power-capital-positif}]
   The proof follows the same steps as the proof of \Cref{thm:capital-power} where we replace the family $(\rme_i)_{i=1}^{2d}$ by $\lrc{1}$.
\end{proof}

\begin{proof}[\bf Proof of \Cref{thm:hoeffding-g-bound}]
Define for all $n \geq 1$, 
$A_n := \lrc{\sum_{t=1}^n g_t(X_t) \geq \sum_{t=1}^n \PEarg[t-1]{g_t(X_t)} - 2\sqrt{n\log(n)}}$
and $B_n := \lrc{\sup_{g\in\cG}\frac{1}{n}\sum_{t=1}^n \PEarg[t-1]{g(X_t)} \geq m_n}
$, so that $\PP[A_n^c] \leq 1/n^2$ by \cite[Lemma~A.7]{Cesa-Bianchi-lugosi-games}, and, by definition, $\varrho = \sum_{n\geq 1} \PP[B_n^c]$.  
Then, letting $G_n := \lrc{\cR_n \leq r_n}\cap\lrc{\cS_n \leq s_n}$, we easily get that, for any $n\geq 1$, we have, on $G_n \cap A_n \cap B_n$,
\begin{align*}
\log W_n^{\rm H} 
\geq \max_{\lambda \in \rset^d} \log L_n^{\rm H}(\lambda) - r_n 
\geq \frac{1}{2n}\lr{\sum_{t=1}^n g_t(X_t)}_+^2 - r_n 
&\geq \frac{1}{2n}\lr{\sum_{t=1}^n \PEarg[t-1]{g_t(X_t)}-2\sqrt{n\log(n)}}_+^2 - r_n \\
&\geq \frac{1}{2n}\lr{nm_n-s_n-2\sqrt{n\log(n)}}_+^2 - r_n \;.
\end{align*}
Hence, letting $E_n := \lrc{\log W_n^{\rm H} \geq u_n}$, we have $\PP[E_n^c\cap G_n \cap A_n \cap B_n] = 0$ for all $n\geq 1$, and therefore
$$
\sum_{n\geq 1} \PP[E_n^c] 
\leq \rho + \varsigma + \varrho + \sum_{n \geq 1} \PP[E_n^c \cap G_n \cap B_n] 
\leq \rho + \varsigma + \varrho + \sum_{n \geq 1} \PP[A_n^c] 
\leq \rho + \varsigma + \varrho + \frac{\pi^2}{6} \;,
$$
which  concludes the proof.
\end{proof}

\begin{proof}[\bf Proof of \Cref{thm:capital-g-bound}]
Let us denote $Z_t = g_t(X_t)$. Using the fact that $\log(1+x) \geq x-x^2$ for any $x \geq  -1/2$, we have, for all $n\geq 1$,
$\log(W_n) \geq \max_{\gamma \in \Gamma} \lr{\sum_{t=1}^n \gamma Z_t - \sum_{t=1}^n (\gamma Z_t)^2}  - \cR_n $.
Define for all $n\geq 1$ and $\gamma \in \Gamma$,
\begin{align*}
B_n(\gamma)&:= \lrc{\sum_{t=1}^n (\gamma Z_t - (\gamma Z_t)^2) \geq \sum_{t=1}^n \PEarg[t-1]{\gamma Z_t} - 4\sum_{t=1}^n \PEarg[t-1]{(\gamma Z_t)^2} - 4 \log(n)}\;,\\
C_n &:= \lrc{\sup_{g\in\cG}\frac{1}{n}\sum_{t=1}^n \PEarg[t-1]{g(X_t)^2} \leq v_n}\;, \\ 
D_n &:= \lrc{\sup_{g\in\cG}\frac{1}{n}\sum_{t=1}^n \PEarg[t-1]{g(X_t)} \geq m_n}\;,
\end{align*}
so that, by \Cref{lem:argument_capital} we have $\PP[B_n(\gamma)^c] \leq 2/n^2$ and $\varrho = \sum_{n\geq 1} \PP[(C_n\cap D_n)^c]$ for all $\gamma \in \Gamma$ and $n \geq 1$. Then, letting $G_n := \lrc{\cR_n \leq r_n}\cap\lrc{\cS_n \leq s_n}$, we have,  for any $\gamma \in [0,1/2] \subset \Gamma$ and $n\geq 1$, on  $G_n\cap C_n\cap D_n \cap  B_n(\gamma)$, 
\begin{align*}
\log(W_n) 
&\geq \gamma \sum_{t=1}^n \PEarg[t-1]{Z_t} - 4\gamma^2 \sum_{t=1}^n \PEarg[t-1]{Z_t^2}  - r_n - 4 \log(n) \\
&\geq \gamma \lr{\sup_{g\in\cG}\sum_{t=1}^n \PEarg[t-1]{g(X_t)} - s_n} - 4\gamma^2 \sup_{g\in\cG}\sum_{t=1}^n \PEarg[t-1]{g(X_t)^2}  - r_n - 4\log(n)  \\
&\geq \gamma  (nm_n - s_n) - 4\gamma^2 n v_n - r_n - 4\log(n) \;.
\end{align*}
Letting $\gamma_n^* := \frac{1}{2} \wedge \frac{(nm_n-s_n)_+}{8nv_n}$, we get that, on $G_n \cap C_n \cap D_n \cap B_n(\gamma_n^*)$, 
$$
\log(W_n) \geq 
\frac{(nm_n-s_n)_+}{4}\lr{1 \wedge \frac{(nm_n-s_n)_+}{4 n v_n} } - r_n  - 4\log(n) = u_n \;.
$$
Hence, letting $E_n := \lrc{\log W_n \geq u_n}$, we have $\PP[E_n^c \cap G_n \cap B_n(\gamma_n^*) \cap C_n \cap D_n] = 0$, for all $n \geq 1$ and therefore
$$
\sum_{n\geq 1} \PP[E_n^c] 
\leq \rho + \varrho + \varsigma + \sum_{n\geq 1} \PP[E_n^c \cap G_n \cap C_n \cap D_n] 
\leq \rho + \varrho + \varsigma  + \sum_{n\geq 1}  \PP[B_n(\gamma_n^*)^c] \leq 
\rho + \varrho + \varsigma + \pi^2/3 \;,
$$
which concludes the proof. 
\end{proof}
\end{document}